\documentclass[12pt,leqno]{amsart}
\usepackage{a4wide}
\usepackage{amssymb}
\usepackage{amsmath}
\usepackage{amsthm}
\usepackage{verbatim}
\usepackage{txfonts}
\usepackage{fullpage}
\usepackage[all,cmtip]{xy} 
\usepackage{yfonts}

\usepackage[T1]{fontenc}
\usepackage[utf8]{inputenc}
\usepackage{mathtools}   
\usepackage{amsfonts}

\numberwithin{equation}{section}

\theoremstyle{plain}

\newcommand{\A}{\ensuremath{{\mathbb{A}}}}
\newcommand{\C}{\ensuremath{{\mathbb{C}}}}
\newcommand{\Z}{\ensuremath{{\mathbb{Z}}}}

\newcommand{\F}{\ensuremath{{\mathbb{F}}}}

\newcommand{\I}{\ensuremath{{\mathbb{I}}}}
\newcommand{\D}{\ensuremath{{\mathbb{D}}}}
\newcommand{\charf}{\textbf{1}}

\newcommand{\GL}{\ensuremath{{\text{GL}}}}
\newtheorem{theo}{Theorem}[section]
\newtheorem{lem}[theo]{Lemma}
\newtheorem{prop}[theo]{Proposition}
\newtheorem{cor}[theo]{Corollary}

\theoremstyle{remark}
\newtheorem{rem}[theo]{Remark}

\theoremstyle{definition}
\newtheorem{defn}[theo]{Definition}

\newtheorem*{cor*}{Corollary}
\newtheorem*{claim1}{Claim 1}
\newtheorem*{claim2}{Claim 2}

\newcommand{\zxz}[4]{\begin{pmatrix} #1 & #2 \\ #3 & #4 \end{pmatrix}}

\newcommand{\Hom}{\operatorname{Hom}}

\title{The subconvexity bound for triple product L-function in level aspect}

\begin{document}
\author{Yueke Hu}

\address{Department of Mathematics, University of Wisconsin Madison, Van Vleck Hall, Madison, WI 53706, USA}
\email{yhu@math.wisc.edu}

\begin{abstract}
In this paper we generalized Venkatesh and Woodbury's work on the subconvexity bound of triple product L-function in level aspect, allowing joint ramifications, higher ramifications, general unitary central characters and general special values of local epsilon factors. In particular we derived a nice general formula for the local integrals whenever one of the representations has sufficiently higher level than the other two.
\end{abstract}
\maketitle

\section{introduction}
Let $\F$ be a number field. Let $\pi_i$, $i=1,2,3$ be three irreducible unitary cuspidal automorphic representations, such that the product of their central characters is trivial:
\begin{equation}
 \prod_{i}w_{\pi_i}=1.
\end{equation}
Let $\Pi=\pi_1\otimes\pi_2\otimes\pi_3$. Then one can define the triple product L-function $L(\Pi, s)$ associated to them.
It was first studied in \cite{Garrett} by Garrett in classical languages, where explicit integral representation was given.
In particular the triple product L-function has analytic continuation and functional equation.
Later on Shapiro and Rallis in \cite{ps} reformulated his work in adelic languages. 

We will consider in this paper the behavior of the special value of triple product L-function $L(\Pi,1/2)$. In particular, we will fix $\pi_1$ and $\pi_2$, let $\pi_3$ vary with finite conductor $\mathcal{N}$. We'd like to study the asymptotic behavior(actually the subconvexity bound) of $L(\Pi,1/2)$ 
as $Nm(\mathcal{N})\rightarrow \infty$.

The idea comes from Venkatesh's work in \cite{AV10}. One starts with the integral representation of the special value of triple product L-function (see, for example, \cite{Ichino}):

\begin{equation}\label{formula1}
 |\int\limits_{Z_{\A}\D^*(\F)\backslash \D^*(\A)} f_1(g)f_2(g)f_3(g) dg|^2=\frac{\zeta_\F^2(2)L(\Pi,1/2)}{8L(\Pi,Ad,1)}\prod_vI_v,
\end{equation}
where $f_i\in \pi_i^{\D}$ for a specific quaternion algebra $\D$, and the local integral $I_v$ can be formulated as follows:
\begin{equation}
I_v=\frac{L_v(\Pi_v,Ad,1)}{\zeta_v^2(2)L_v(\Pi_v,1/2)}\int\limits_{\F_v^*\backslash \D^*(\F_v)}\prod\limits_{i=1}^{3}<\pi^{\D}_i(g)f_{i,v},f_{i,v}>dg.
\end{equation}
Here $<\cdot,\cdot>$ is a bilinear and $\D^*(\F_v )-$ invariant unitary pairing for $\pi^{\D}_{i,v}$.
At unramified places, this local integral is 1.

Suppose now the cusp forms and their local components are properly normalized. The idea in \cite{AV10} is to give first an upper bound for the left-hand side of (\ref{formula1}). 
Then a lower bound for the local integrals $I_v$ will result in an upper bound for $L(\Pi,1/2)$, which turn out to be a subconvexity bound in the level aspect. In particular, assume that $\pi_3$ is of prime conductor $\textswab{p}$.
Venkatesh's work together with Woodbury's work on local integrals in \cite{MW12} prove the following:
\begin{equation}\label{Mikessubconv}
L(\Pi,1/2)<<N(\textswab{p})^{1-1/12}.
\end{equation}
Note that the trivial bound for the triple product L-function is when the power is $1$. Any result with power less than 1 counts as a subconvexity bound.

Their result, however, is based on the following conditions:
\begin{enumerate}
 \item[(1)]$\pi_i$ essentially have disjoint ramifications and $\pi_3$ has square-free finite conductor $\textswab{p}$.
 \item[(2)]All the central characters are trivial.
 \item[(3)]The special values of local epsilon factors $\epsilon_v(\Pi,1/2)=1$ for all places.
 \item[(4)]The infinity component of $\pi_3$ is bounded. 
\end{enumerate}

In this paper, we will remove the first three conditions and prove a similar subconvexity bound. So we will allow high ramifications and joint ramifications, and general unitary central characters. The third condition is related to Prasad's thesis work on local trilinear forms, and turns out to be free to remove. This is because, as we will see later,  all key calculations will be done on the $\GL_2$ side.
The last condition is still necessary as it is used to control $L(\Pi,Ad,1)$. Then we will prove in Theorem \ref{thmmain} that for fixed $\pi_1$ and $\pi_2$, and $\pi_3$ with changing finite conductor $\mathcal{N}$, 
\begin{equation}\label{formula2}
 L(\Pi,1/2)<<\text{Nm}(\mathcal{N})^{1-1/12}.
\end{equation}


We shall follow the same strategy. In Section 2 we will review necessary tools and results, as well as derive some new results which will be used in this paper. In Section 3, we basically imitate Venkatesh's proof and get an upper bound for the global integral in more general setting. We will use amplication method and reduce the problem to a bound for global matrix coefficient. 
In Section 4 we will derive the lower bound for local integrals by explicit computations. Assume that $c_3\geq 2\max\{c_1,c_2,1\}$, where $c_i$ is the local level of $\pi_i$ at a finite place $v$. Let $\Phi_i(g)$ be the local matrix coefficients associated to certain elements in $\pi_{i,v}$ to be specified later. Then Theorem \ref{thmlocalint} shows that
\begin{equation}
\int\limits_{\F_v^*\backslash \GL_2(\F_v)}\prod\limits_{i=1}^{3}\Phi_i(g)dg=\frac{(1-A)(1-B)}{(q+1)q^{c_3-1}},
\end{equation}
where $A$ and $B$ are fixed values only depending on $\pi_{1,v}$ and $\pi_{2,v}$. One can further check case by case and show that $A$ and $B$ are bounded away from 1 using the bound towards Ramanujan conjecture.

Before this paper, there is little work on explicit computation for the local integral with ramifications. Woodbury in \cite{MW12} considered the special unramified representations.  In \cite{NPS12}, Nelson, Pitale and Saha computed $I_v$ for higher ramifications, essentially with the assumption that $\pi_1=\pi_2$ (and correspondingly $f_1=f_2$) and $\pi_3$ is unramified.
Their work is based on Lemma (3.4.2) of \cite{MV10}, which relates $I_v$ to the local Rankin-Selberg integral. But this method can't be generalized to the case when all the representations are supercuspidal, which is necessary for our consideration. Their result is given case-by-case, and is quite complicated. So it's quite surprising that in our setting we can get such a simple and nice formula.

In Section 5 we will finish the proof of (\ref{formula2}). In the appendix we will prove the bound for the global matrix coefficient which is used in the proof in Section 3.
\section{Notations and preliminary results}

\subsection{Basic Notations and facts}
Let $\F$ denote a number field. Let $G$ be a reductive algebraic $\F-$group. In this paper we will focus on $G$ being $\GL_2$ or $\D^*$, where $\D$ is a quaternion algebra. 
Let $X=Z_G(\A)G(\F)\backslash G(\A_\F)$. Let $L^2(X)$ be the space of square integrable functions on X, and $<\cdot,\cdot>$ be the natural pairing on it given by 
\begin{equation}
 <f_1,f_2>=\int_Xf_1(g)\overline{f_2(g)}dg.
\end{equation}
Any unitary cuspidal automorphic representation can be naturally embedded into $L^2(X)$ with the compatible unitary pairings.

Let $\F_v$ be the corresponding local field of $\F$ at a place $v$. Let $K_v$ denote the maximal compact subgroup of $G(\F_v)$, and 
\begin{equation}
 K=\prod_vK_v.
\end{equation}
When $v$ is a finite place, let $\varpi_v$ denote a uniformizer of $\F_v$ and $O_{v}$ denote
the ring of integers at $v$. Let $q^{-1}=|\varpi_v|_v$.
Define for an integer $c>0$
\begin{equation}
 K_1(\varpi_v^c)=\{k\in K_v| \text{\ }k\equiv \zxz{*}{*}{0}{1} \mod{(\varpi_v^c)}\}.
\end{equation}


Now we record some basic facts about integrals on $\GL_2(\F_v)$. 
\begin{lem} \label{Iwasawadecomp}
For every positive integer $c$,
$$\GL_2(F_v)=\coprod\limits_{0\leq i\leq c} B\zxz{1}{0}{\varpi_v^i}{1}K_1(\varpi_v^c).$$
Here $B$ is the Borel subgroup of $\GL_2$. 
\end{lem}
We normalize the Haar measure on $\GL_2(\F_v)$ such that $K_v$ has volume 1. Then we have the following easy result (see, for example, \cite[Appendix A]{YH13}).
\begin{lem}\label{localintcoefficient}
Locally let $f$ be a $K_1(\varpi_v^c)-$invariant function, on which the center acts trivially. Then
 \begin{equation}
  \int\limits_{F_v^*\backslash\GL_2(\F_v)}f(g)dg=\sum\limits_{0\leq i\leq c}A_i\int\limits_{\F_v^*\backslash B(\F_v)}f(b\zxz{1}{0}{\varpi_v^i}{1})db.
 \end{equation}
 
Here $db$ is the left Haar measure on $\F_v^*\backslash B(\F_v)$, and
$$A_0=\frac{q}{q+1}\text{,\ \ \ }A_c=\frac{1}{(q+1)q^{c-1}}\text{,\ \ \  and\ }A_i=\frac{q-1}{(q+1)q^i}\text{\ for\ }0<i<c.$$
\end{lem}

\subsection{Integral representation of special values of Triple product $L-$function}
The story begins with Prasad's thesis work. For the triple product L-function $L(\Pi,s)$, there exist local epsilon factors $\epsilon_v(\Pi_v,\psi_v,s)$ and global epsilon factor $\epsilon(\Pi,s)=\prod_v\epsilon(\Pi_v,\psi_v,s)$, such that,
\begin{equation}
 L(\Pi,1-s)=\epsilon(\Pi,s)L(\check{\Pi},s).
\end{equation}
With the assumption that $\prod_i w_{\pi_i}=1$, we have $$\Pi\cong\check{\Pi}.$$ The special values of local epsilon factors $\epsilon_v(\Pi_v,\psi_v,1/2)$ are actually independent of $\psi_v$ and always take value $\pm 1$. For simplicity, we will write 
$$\epsilon_v(\Pi_v,1/2)=\epsilon_v(\Pi_v,\psi_v,1/2).$$ For any place $v$, there is a unique (up to isomorphism) division algebra $\D_v$. Then
Prasad proved in \cite{Prasad} the following theorem about the dimension of the space of local trilinear forms:
\begin{theo}
\begin{enumerate}
 \item $\dim\Hom_{\GL_2(\F_v)}(\Pi_v,\C)\leq 1$, with the equality if and only if $\epsilon_v(\Pi_v,1/2)=1$.
 \item $\dim\Hom_{\D_v}(\Pi^{\D_v}_v,\C)\leq 1$, with the equality if and only if $\epsilon_v(\Pi_v,1/2)=-1$.
\end{enumerate}
Here $\Pi^{\D_v}_v$ is the image of $\Pi_v$ under Jacquet-Langlands correspondence.
\end{theo}
This motivated the following result which is conjectured by Jacquet and later on proved by Harris and Kudla in \cite{H&K91} and \cite{H&K04}:
\begin{theo}\label{thmofJacquetconj}
 $$\{L(\Pi,1/2)\neq 0\}  \Longleftrightarrow \left\{ \begin{array}{c}
                                                    \text{ there exist\ } \D\text{\ and\ } f_i\in \pi^{\D} \text{\ s.t.}\\
                                                    \int\limits_{Z_{\A}\D^*(\F)\backslash \D^*(\A)} f_1(g)f_2(g)f_3(g) dg\neq 0
                                                   \end{array}\right\}
 $$
\end{theo}
This result hints that  $$\int\limits_{Z_{\A}\D^*(\F)\backslash \D^*(\A)} f_1(g)f_2(g)f_3(g) dg$$ could be a potential integral representation of special value of triple product L-function. Later on there are a lot of work on explicitly relating both sides. In particular one can see Ichino's work in \cite{Ichino}. We only need a special version here (as in the introduction).

\begin{equation}\label{Globaltriple}
 |\int\limits_{Z_{\A}\D^*(\F)\backslash \D^*(\A)} f_1(g)f_2(g)f_3(g) dg|^2=\frac{\zeta_\F^2(2)L(\Pi,1/2)}{8L(\Pi,Ad,1)}\prod_vI_v,
\end{equation}
where $f_i\in \pi_i^{\D}$ for the specific quaternion algebra $\D$ as in the theorem above, and the local integral $I_v$ can be formulated as follows:
\begin{equation}\label{localtripleF}
I_v(f_1,f_2,f_3)=\frac{L_v(\Pi_v,Ad,1)}{\zeta_v^2(2)L_v(\Pi_v,1/2)}\int\limits_{\F_v^*\backslash \D^*(\F_v)}\prod\limits_{i=1}^{3}<\pi^{\D}_i(g)f_{i,v},f_{i,v}>dg.
\end{equation}
Here $<\cdot,\cdot>$ is a bilinear and $\D^*(\F_v )-$ invariant unitary pairing for $\pi^{\D}_{i,v}$.
At unramified places, this local integral is 1.

\subsection{Hecke operators} For the beginning of this subsection one can also see \cite{AV10}.
Let $f$ be a function on a group $G$ and $\sigma $ a compactly supported measure on $G$. Define the convolution of $f$ with $\sigma$ by 
\begin{equation}
 f*\sigma(x)=\int\limits_{g}f(xg)d\sigma(g).
\end{equation}
If $\sigma_1$ and $\sigma_2$ are two compactly supported measures on $G$, we define the convolution $\sigma_1*\sigma_2$ to be the pushforward to $G$ of $\sigma_1\times\sigma_2$ on $G\times G$, under the multiplication map 
\begin{equation}
 (g_1,g_2)\in G\times G\mapsto g_1g_2.
\end{equation}
Then one has the following compatibility relation
\begin{equation}
 (f*\sigma_2)*\sigma_1=f*(\sigma_1*\sigma_2).
\end{equation}
Now we introduce the Hecke operators in this language. At a non-archimedean place $v$, let $\mathfrak{l}$ be a maximal prime ideal and $r$ be an integer $\geq 0$. Define the measure $\mu^*_{\mathfrak{l}^r}$ on $\GL_2(\F_v)$ to be the restriction of Haar measure to the set 
$$K\zxz{\varpi^r}{0}{0}{1}K,$$
so that the total mass of $\mu^*_{\mathfrak{l}^r}$ is $\begin{cases}
                                                  (q+1)q^{r-1}, &\text{\ if\ }r\geq 1;\\
                                                  1, &\text{\ if\ }r=0.
                                                 \end{cases}$

Define
\begin{equation}
 \mu_{\mathfrak{l}^r}=\frac{1}{q^{r/2}}\sum\limits_{0\leq k\leq r/2}\mu^*_{\mathfrak{l}^{r-2k}}.
\end{equation}
Via the natural inclusion of $\GL_2(\F_v)$ in $\GL_2(\A_{\F,f})$, we can regard $ \mu_{\mathfrak{l}^r}$ as a compactly supported measure on $\GL_2(\A_{\F,f})$. If $\mathfrak{n}$ is an integral ideal $\prod_v{\mathfrak{l}_v^{r_v}}$, define
\begin{equation}
 \mu_{\mathfrak{n}}=\prod_v\mu_{\mathfrak{l}_v^{r_v}}.
\end{equation}
Convolution by $\mu_\mathfrak{n}$ can be thought of as $\mathfrak{n}-$th Hecke operator.

For functions on which the center acts trivially, convolution with $\mu_\mathfrak{n}$ is a self-dual operator, that is,
\begin{equation}\label{formulaofselfdualHecke}
\int\limits_{G(\A)} f_1\cdot (f_2*\mu_\mathfrak{n})dg=\int\limits_{G(\A)} (f_1*\mu_\mathfrak{n})\cdot f_2 dg.
\end{equation}
Similarly one can see that 
\begin{equation}
\int\limits_{G(\A)}f(xg)d\mu_\mathfrak{n}(g)=\int\limits_{G(\A)}f(xg^{-1})d\mu_\mathfrak{n}(g).
\end{equation}
Further we have the following nice lemma about compositions of Hecke operators:
\begin{lem}\label{convolutionofHecke}
 Let $\mathfrak{n}$, $\mathfrak{m}$ be ideals. Let $h$ be a function on $G(\A_\F)$ that is spherical at all places $v|\mathfrak{n}\mathfrak{m}$, and the center acts on $h$ trivially. Then 
 \begin{equation}
  \int_{G(\A_\F)}h(x)d(\mu_\mathfrak{n}*\mu_\mathfrak{m})(x)=\sum_{\mathfrak{d}|(\mathfrak{n},\mathfrak{m})}\int_{G(\A_\F)}h(x)d\mu_{\mathfrak{n}\mathfrak{m}\mathfrak{d}^{-2}}(x). \end{equation}
\end{lem}

We will also need to consider, however, functions on which the center acts by a non-trivial unitary character $w$. From now on we will only consider operators of form $\mu_\mathfrak{l}$ or $\mu_{\mathfrak{l}^2}$ at a finite place $v$. Let $\check{\mu_\mathfrak{l}}$ be the dual of $\mu_\mathfrak{l}$ in the sense of (\ref{formulaofselfdualHecke}). Then one can easily check that
\begin{equation}\label{Hecke1}
\check{\mu_\mathfrak{l}}=w(\varpi_v^{-1})\mu_\mathfrak{l}=\frac{w(\varpi_v^{-1})}{\sqrt{q}}\mu_\mathfrak{l}^*.
\end{equation}
Similarly let $\check{\mu_{\mathfrak{l}^2}}$ be the dual of $\mu_{\mathfrak{l}^2}$. Then
\begin{equation}\label{Hecke2}
\check{\mu_{\mathfrak{l}^2}}=\frac{1}{q}(w(\varpi_v^{-2})\mu_{\mathfrak{l}^2}^*+\mu_\mathfrak{o}^*).
\end{equation}
When acting on spherical functions, $\mu_\mathfrak{l}^*$ and $\mu_{\mathfrak{l}^2}^*$ are related as follows:
\begin{equation}\label{Hecke3}
\mu_\mathfrak{l}^**\mu_\mathfrak{l}^*=\mu_{\mathfrak{l}^2}^*+(q+1)w(\varpi_v)\mu_\mathfrak{o}^*.
\end{equation}
Now let $\check{\lambda_\mathfrak{l}}$ and $\check{\lambda_{\mathfrak{l}^2}}$ be the eigenvalues of $\check{\mu_\mathfrak{l}}$ and $\check{\mu_{\mathfrak{l}^2}}$ acting on a given spherical function. Putting (\ref{Hecke1}), (\ref{Hecke2}) and (\ref{Hecke3}) together, we have
\begin{equation}
\check{\lambda_{\mathfrak{l}^2}}=\check{\lambda_\mathfrak{l}}^2+(q^{-1}-\frac{q+1}{qw(\varpi_v)}).
\end{equation}
Note 
$$|q^{-1}-\frac{q+1}{qw(\varpi_v)}|\geq 1.$$ 
Then one can easily check that,
\begin{cor}\label{corofboundoneigenvalues}
$$|\check{\lambda_{\mathfrak{l}^2}}|+|\check{\lambda_\mathfrak{l}}|\geq 1.$$
\end{cor}

\subsection{Bounds for matrix coefficient}
If $\pi$ is an irreducible unitary cuspidal automorphic representation, then its local component at $v$ is also a unitary representation. At a non-archimedean place, it can be classified into one of the following four types:
\begin{enumerate}
 \item supercuspidal representation;
 \item $\pi(\chi_1,\chi_2)$ where $\chi_i$ are unitary characters;
 \item special representation $\sigma(\chi|\cdot|^{1/2},\chi|\cdot|^{-1/2})$ where $\chi$ is unitary;
 \item $\pi(\chi|\cdot|^{\tau},\chi|\cdot|^{-\tau})$, where $\chi$ is unitary and $0<\tau<1/2$.
\end{enumerate}
The first three types are tempered representations. The generalized Ramanujan Conjecture implies that only tempered representations can be the local component of a unitary cuspidal automorphic representation. What is known is a bound $\alpha$ towards Ramanujan conjecture. This means if type (4) ever happens, then $\tau<\alpha$.
The smaller $\alpha$ is, the closer we are to the Ramanujan Conjecture for $\GL_2$. For our purpose, any $\alpha<1/4$ would be enough to get a subconvexity bound. The current record is $\alpha=7/64$. See \cite{Kim}, \cite{BB10}.

Using the bound towards Ramanujan Conjecture, one can bound the matrix coefficient for the local component of a unitary cuspidal automorphic representation.

Locally for $f_1\in\pi_v\cong\pi(\chi_1,\chi_2)$, $f_2\in \check{\pi_v}\cong\pi(\chi_1^{-1},\chi_2^{-1})$ in the standard model for the induced representations, we can define the pairing by
\begin{equation}
 <f_1,f_2>=\int\limits_{K_v}f_1(k)f_2(k)dk.
\end{equation}
We can define the matrix coefficient of $\pi_v$ associated to $f_1$, $f_2$ as
\begin{equation}
\Phi(g)= <\pi_v(g)f_1,f_2>.
\end{equation}

See later subsections for the alternative definition and the definition when the representation is supercuspidal.

 We first record here the matrix coefficient for spherical elements. (See for example, \cite{Bump}.) For simplicity, let $\chi_i$ denote $\chi_i(\varpi_v)$ in the following formulae if we don't specify which element the characters are taking.
\begin{lem}\label{lemofMCforUNram}
 Let $\pi=\pi(\chi_1,\chi_2)$ be an unramified unitary representation of $\GL_2$. Let $\Phi$ be the matrix coefficient associated to normalized newforms in $\pi$. Then it's bi-$K-$invariant and 
 \begin{equation}
  \Phi(\zxz{\varpi_v^n}{0}{0}{1})=\frac{q^{-n/2}}{1+q^{-1}}\frac{\chi_1^n(\chi_1-\chi_2q^{-1})-\chi_2^n(\chi_2-\chi_1q^{-1})}{\chi_1-\chi_2}.
 \end{equation}

\end{lem}

Now we state the result for the bound of local matrix coefficient for general elements. (See for example,  .)
\begin{lem}\label{localboundmatrixcoeff}
 Let $\pi_v$ be the local component of an unitary cuspidal automorhpic representation of $\GL_2$ at a finite place $v$. Let $f_1$, $f_2$ be two $K_v-$finite elements in $\pi_v$, stabilized respectively by compact open subgroups $K_{1,v}$ and $K_{2,v}$. 
 Then for any
 $x\in\F_v$ and $\epsilon>0$, 
 \begin{equation}
  |<\pi(\zxz{x}{0}{0}{1})f_1,f_2>|\ll_{\epsilon,\F}[K_v:K_{1,v}]^{1/2}[K_v:K_{2,v}]^{1/2}q^{(\alpha-1/2+\epsilon)|v(x)|}||f_1||_v||f_2||_v.
 \end{equation}
 \end{lem}

\begin{proof}
 It follows from, for example, Lemma 9.1 of \cite{AV10}. Here we briefly describe how to prove this result for induced representations at non-archimedean places. For spherical elements, one can use Lemma \ref{lemofMCforUNram} above to check the inequality directly. More specifically if $|v(x)|=n$ and $f_i$'s are spherical, then
 \begin{align}
 |<\pi(\zxz{x}{0}{0}{1})f_1,f_2>|&=|\frac{q^{-n/2}}{1+q^{-1}}\frac{\chi_1^n(\chi_1-\chi_2q^{-1})-\chi_2^n(\chi_2-\chi_1q^{-1})}{\chi_1-\chi_2}|\\
 &=|\frac{q^{-n/2}}{1+q^{-1}}((\chi_1^n+\chi_1^{n-1}\chi_2+\cdots+\chi_2^{n}) -q^{-1}\chi_1\chi_2(\chi_1^{n-2}+\chi_1^{n-3}\chi_2+\cdots+\chi_2^{n-2}))|\notag\\
 &\leq (n+1)q^{(\alpha-1/2)n}.\notag
 \end{align}
 The coefficient $(n+1)$ will be essentially bounded by $q^{\epsilon n}$ for any $\epsilon> 0$, and the implicit constant can be taken to be 1 when $q$ is large enough.
 When $f_1$, $f_2$ are not spherical, one can use the trick as in \cite{CHH88} to reduce the inequality to the spherical case.
\end{proof}
 
\begin{rem}\label{remoftakingproductofbound}
  This proof actually allow one to control the implicit constant. In particular one can take a product of the local inequality and get a global inequality.
\end{rem}

Now we give a bound for the global matrix coefficient. Let $\D$ be a global quaternion algebra. Let $\rho$ denote the right regular representation of $\D^*(\A)$ on $L^2(Z_{\A}\D^*(\F)\backslash \D^*(\A))$. 
Let $F_1$, $F_2\in L^2(Z_{\A}\D^*(\F)\backslash \D^*(\A))$ be two rapidly decreasing and $K-$finite automorphic forms which don't have 1-dim components in their spectrum decomposition.
 Let $S$ be a finite set of non-archimedean places. We assume that $\D$ is locally the matrix algebra at the places in $S$.  
 Let $K_S=\prod_{v\in S}K_v$ and $K_{i,S}=\prod_{v\in S}K_{i,v}$, where $K_{i,v}$ stabilizes the local component of $F_i$ at $v$. Let $\mathcal{N}=\prod_v\varpi_v^{e_v}$ for $e_v\geq 0$, and $N=\text{Nm}(\mathcal{N})$. Define the matrix 
 $$a([\mathcal{N}])=\prod_v\zxz{\varpi^{-e_v}}{0}{0}{1},$$
 which can be naturally thought of as an element of $\D^*(\A)$.
\begin{prop}\label{propofboundglobalMC}
With the setting as above, we have
 \begin{equation}
  |\int\limits_{Z_{\A}\D^*(\F)\backslash \D^*(\A)}F_1(g)\rho(a([\mathcal{N}]))F_2(g)dg|\ll_{\epsilon,\F}[K_S:K_{1,S}]^{1/2}[K_S:K_{2,S}]^{1/2}N^{\alpha-1/2+\epsilon}||F_1||_{L^2}||F_2||_{L^2}.
 \end{equation}

\end{prop}
We will prove this proposition in the appendix.
Now the question is, for any given cusp forms $F$, how can we separate out the 1-dimensional components. Suppose  that in general the center acts on $F$ by a unitary central character $w$. Then its 1-dimensional components can be given as the following projection:
\begin{equation}
F\mapsto  \mathcal{P}F(x)=\sum_{\chi^2=w}\chi(x)\int_{X}f(y)\overline{\chi(y)}dy.
\end{equation}
Then the remaining part $F-\mathcal{P}F$ doesn't have any 1-dimensional components.

\subsection{Whittaker model for induced representations} \label{SecprepofWhittaker}
Here we recall some basic results about the Whittaker model for induced representations.  This and next subsections are purely local, so we will suppress the subscript $v$ for all notations.

Fix an additive character $\psi$. Without loss of generality, we will always assume $\psi$ is unramified.
Let $\pi$ be a local irreducible (generic) representation of $G$. Then there is a unique realization of $\pi$ in the space of functions $W$ on $G$ such that
\begin{equation}
W(\zxz{1}{n}{0}{1}g)=\psi(n)W(g).
\end{equation}
Locally for an induced representation of $\GL_2$, one can compute its Whittaker functional by the following formula:
\begin{equation}\label{computeWhit}
W (g)=\int\limits_{m\in \F }\varphi(\omega\zxz{1}{m}{0}{1}g)\psi(m)dm,
\end{equation}
where $\varphi$ is an element of $\pi$ in the model of induced representation and $\omega$ is the matrix $\zxz{0}{1}{-1}{0}$.

When $\pi$ is unitary, one can define a unitary pairing on $\pi$ using the Whittaker model:
\begin{equation}
<W_1,W_2>=\int_{\F^*}W_1(\zxz{\alpha}{0}{0}{1})\overline{W_2(\zxz{\alpha}{0}{0}{1})}d^*\alpha.
\end{equation}

To get the Whittaker functional explicitly using (\ref{computeWhit}),  the first step is to write $$\omega\zxz{1}{m}{0}{1}\zxz{\alpha}{0}{0}{1}\zxz{1}{0}{\varpi^i}{1}=\zxz{\varpi^i}{1}{-\alpha-m\varpi^i}{-m}$$ 
in form of $B(\F )\zxz{1}{0}{\varpi^k}{1}K_1(\varpi^c)$ for $0\leq i,k \leq c$. Note that if $k=c$, then $\zxz{1}{0}{\varpi^k}{1}$ is absorbed into $K_1(\varpi^c)$. Same for $i$.

We record the following results about from \cite{YH13}.
\begin{lem}\label{LevelIwaDec}
\begin{enumerate}
\item[(1)]Suppose $k=0$.
\begin{enumerate}
\item[(1i)]If $i=0$, we need $m\notin \alpha(-1+\varpi O_F)$ for  $\zxz{\varpi^i}{1}{-\alpha-m\varpi^i}{-m}\in B\zxz{1}{0}{\varpi^k}{1}K_1(\varpi^c)$;
\item[(1ii)]If $i>0$, we need $v(m)\geq v(\alpha)$. 

\end{enumerate}
Under above conditions we can write $\zxz{\varpi^i}{1}{-\alpha-m\varpi^i}{-m}$ as
\begin{equation*}
\zxz{-\frac{\alpha}{\alpha+m\varpi^i}}{\varpi^i+\frac{\alpha}{\alpha+m\varpi^i}}{0}{-\alpha-m\varpi^i}\zxz{1}{0}{1}{1}\zxz{1}{-1+\frac{m}{\alpha+m\varpi^i}}{0}{1}.
\end{equation*}
\item[(2)]Suppose $k=c$.
\begin{enumerate}
\item[(2i)]If $i<c$, we need $m\in \alpha\varpi^{-i}(-1+\varpi^{c-i}O_F)$;
\item[(2ii)]If $i=c$, we need $v(m)\leq v(\alpha)-c$.
\end{enumerate}
Under above conditions, we can write $\zxz{\varpi^i}{1}{-\alpha-m\varpi^i}{-m}$ as
\begin{equation*}
\zxz{-\frac{\alpha}{m}}{1}{0}{-m}\zxz{1}{0}{\frac{\alpha}{m}+\varpi^i}{1}.
\end{equation*}

\item[(3)]Suppose $0<k<c$.
\begin{enumerate}
\item[(3i)]If $i<k$, we need $m\in \alpha\varpi^{-i}(-1+\varpi^{k-i}O_F^*)$;
\item[(3ii)]If $i>k$, we need $v(m)=v(\alpha)-k$;
\item[(3iii)]If $i=k$, we need $v(m)\leq v(\alpha)-k$ but $m\notin \alpha\varpi^{-k}(-1+\varpi O_F)$. 
\end{enumerate}
Under above conditions we can write $\zxz{\varpi^i}{1}{-\alpha-m\varpi^i}{-m}$ as
\begin{equation*}
\zxz{-\frac{\alpha\varpi^k}{\alpha+m\varpi^i}}{1}{0}{-m}\zxz{1}{0}{\varpi^k}{1}\zxz{\frac{\alpha+m\varpi^i}{m\varpi^k}}{0}{0}{1}.
\end{equation*}
\end{enumerate}
\end{lem}
\begin{proof}
Direct to check.
\end{proof}

Now let $\pi$ be a unitary induced representation $\pi(\mu_1,\mu_2)$, where $\mu_1$ and $\mu_2$ are both ramified of level $k_1$ and $k_2$. Let $c=k_1+k_2$ be the level of $\pi$. Then by the classical results, there exists a newform in the model of induced representation, which is right $K_1(\varpi^c)-$invariant and supported on
$$B\zxz{1}{0}{\varpi^{k_2}}{1}K_1(\varpi^c),$$
where $B$ is the Borel subgroup.

We shall consider the Whittaker function $W$ associated to this newform.
Let 
\begin{equation}
C=\int\limits_{u\in O_F^*} \mu_1(-\varpi^{k_2})\mu_2(-\varpi^{-k_2} u)\psi(-\varpi^{-k_2} u)du.
\end{equation}
We denote the normalized Whittaker value $W (\zxz{\alpha}{0}{0}{1}\zxz{1}{0}{\varpi^i}{1})$ by $W^{(i)}(\alpha)$ for short. Then the next lemma follows directly from (\ref{computeWhit}) and (3) of the above lemma.
\begin{lem}\label{Wiofram}

\begin{enumerate}
\item[(i)]If $i<k_2$, then 
\begin{equation}
W^{(i)}(\alpha)=C^{-1}\int\limits_{u\in O_F^*}\mu_1(-\frac{\varpi^i}{u})\mu_2(\alpha\varpi^{-i}(1-\varpi^{k_2-i}u))q^{v(\alpha)/2-i}\psi(\alpha\varpi^{-i}(1-\varpi^{k_2-i}u))q^{2i-k_2-v(\alpha)}du.
\end{equation}
Its integral against 1 is always 0. 

\item[(ii)]If $k_2<i\leq c$, then 
\begin{equation}
W^{(i)}(\alpha)=C^{-1}\int\limits_{u\in O_F^*} \mu_1(-\frac{\varpi^{k_2}}{1+u\varpi^{i-k_2}})\mu_2(-\varpi^{-k_2}\alpha u)q^{-v(\alpha)/2}\psi(-\varpi^{-k_2}\alpha u)du.
\end{equation}
In particular 
\begin{equation}
W^{(c)}(\alpha)=\begin{cases}
1,&\text{\ if\ }v(\alpha)=0;\\
0,&\text{\ otherwise.}
\end{cases}.
\end{equation}
When $i<c$,
\begin{equation}\label{RamWint1}
\text{\ \ \ }\int\limits_{v(\alpha) \text{fixed}}W^{(i)}(\alpha)d^*\alpha=\begin{cases}
-\frac{1}{q-1},&\text{\ \ if\ }i=c-1>k_2\text{\ and\ }v(\alpha)=0;\\
0,&\text{\ \ otherwise}.
\end{cases}
\end{equation}
\item[(iii)]If $i=k_2$, 
\begin{equation}
W^{(k_2)}=C^{-1}\int\limits_{v(u)\leq -k_2,u\notin \varpi^{-k_2}(-1+\varpi O_F)}\mu_1(-\frac{\varpi^{k_2}}{1+u\varpi^{k_2}})\mu_2(-\alpha u)|\frac{\varpi^{k_2}}{\alpha u(1+u\varpi^{k_2})}|^{1/2}\psi(-\alpha u)q^{-v(\alpha)}du.
\end{equation}
The integral of $W^{(k_2)}$ against 1 is always zero if either $k_2>1$ or $v(\alpha)\neq 0$. When $k_2=1$ and $v(\alpha)=0$, its integral against 1 is the same as expected from (2) as the limit case.
\end{enumerate}
\end{lem}

We shall also consider the case when $\pi\cong\pi(\mu_1,\mu_2)$, where $\mu_1$ is unramified and $\mu_2$ is ramified of level $k$. Then the level of the representation $\pi$ is $k$. In this case the newform is right $K_1(\varpi^k)-$invariant and supported on $BK_1(\varpi^k)$.
Then by (2) of Lemma \ref{LevelIwaDec}, we have
\begin{lem}\label{Wiofbiasram}
\begin{enumerate}
\item[(1)]When $i=k$, 
\begin{align}
W^{(k)}(\alpha)&=\int\limits_{v(m)\leq v(\alpha)-k}\mu_1(-\frac{\alpha}{m})\mu_2(-m)\psi(-m)q^{-\frac{1}{2}v(\alpha)+v(m)}dm\\
&=\begin{cases}
q^{-\frac{1}{2}v(\alpha)}\mu_1^k(\varpi)q^{-k}\int\limits_{v(m)=-k}\mu_2(-m)\psi(-m)dm, &\text{\ if\ }v(\alpha)\geq 0,\\
0,& \text{\ otherwise.}
\end{cases}\notag
\end{align}
\item[(2)]When $i<k$, 
\begin{equation}
W^{(i)}(\alpha)=\mu_1^i(\varpi)\int\limits_{u\in O_\F}\mu_2(\alpha\varpi^{-i}(1-\varpi^{k-i}u))\psi(\alpha\varpi^{-i}(1-\varpi^{k-i}u))q^{-\frac{1}{2}v(\alpha)-k+i}du.
\end{equation}
\end{enumerate}

\end{lem}
\begin{rem}
In this lemma, the Whittaker functional is not normalized. But this turns out to be enough.
\end{rem}

\subsection{Kirillov model for supercuspidal representations}
Now let's consider supercuspidal representations. For the fixed additive character $\psi$, the Kirillov model of $\pi$ is a unique realization on $S(\F ^*)$ such that
\begin{equation}\label{Kirilmodel}
\pi(\zxz{a_1}{m}{0}{a_2})\varphi(x)=w_{\pi}(a_2)\psi(ma_2^{-1}x)\varphi(a_1a_2^{-1}x),
\end{equation}
where $w_{\pi}$ is the central character for $\pi$. Note if $\pi$ is not supercuspidal, one can still define its Kirillov model, but it's realized in $S(\F)$. Let $W_\varphi$ be the Whittaker function associated to $\varphi$. Then they are related by
$$\varphi(\alpha)=W_\varphi(\zxz{\alpha}{0}{0}{1}),$$
$$W_\varphi(g)=\pi(g)\varphi(1).$$
When $\pi$ is unitary, one can define the $G-$invariant unitary pairing on Kirillov model by
\begin{equation}
<f_1,f_2>=\int\limits_{\F ^*}f_1(x)\overline{f_2}(x)d^*x.
\end{equation}

By Bruhat decompostion, one just has to know the action of $\omega=\zxz{0}{1}{-1}{0}$ to understand the whole group action. 

Define $$\charf_{\nu,n}(x)=\begin{cases}
                        \nu(u), &\text{if\ } x=u\varpi^n\text{\  for\ } u\in O_F^*;\\
						0,&\text{otherwise}.
                       \end{cases}
 $$ Roughly speaking, it's the character $\nu$ supported at $v(x)=n$.
We can then describe the action of $\omega=\zxz{0}{1}{-1}{0}$ on $\charf_{\nu,n}$ explicitly according to \cite{JL70}:

\begin{equation}\label{singleaction}
\pi(\omega)\charf_{\nu,n}=C_{\nu w_0^{-1}}z_0^{-n}\charf_{\nu^{-1}w_0,-n+n_{\nu^{-1}}}.
\end{equation}
Here $z_0=w(\varpi)$ and $w_0=w_\pi|_{O_F^*}$. 
It's well-known that $n_{\nu}\leq -2$ for any $\nu$.
$-n_1$ is actually the level of this supercuspidal representation. Denote $c=-n_1$. 
The corresponding newform is simply $\charf_{1,0}$.

The relation $\omega^2=-\zxz{1}{0}{0}{1}$ implies 
\begin{equation}
n_{\nu}=n_{\nu^{-1}w_0^{-1}},\text{\ \ } C_\nu C_{\nu^{-1}w_0^{-1}}=w_0(-1)z_0^{n_\nu}.
\end{equation}

\begin{rem}\label{remrelatingscandlevel}
According to \cite{Yo77}, another way to formulate (\ref{singleaction}) is
 $$\pi(\omega)\charf_{\lambda_0,n}=\epsilon(\pi\otimes\lambda^{-1},\psi,1/2)\charf_{\lambda_0^{-1}w_0,-n-c(\pi\otimes\lambda^{-1})},$$
 where $\lambda$ is a character of $\F^*$ and $\lambda_0=\lambda|_{O_F^*}$. $c(\pi\otimes\lambda^{-1})$ is the level of $\pi\otimes\lambda^{-1}$. In particular, this implies $n_{\lambda_0}=-c(\pi\otimes\lambda^{-1})$. Also $C_\nu$'s are related to special values of local epsilon factors.
\end{rem}

It is proved in \cite[Proposition B.3]{YH13} that

\begin{prop}\label{Propofnmu}
Suppose that $c=-n_1\geq 2$ is the level of a supercuspidal representation $\pi$ whose central character is unramified or level 1. If $p\neq 2$ and $\nu$ is a level $i$ character, then we have $$n_\nu=\min\{-c,-2i\}.$$ 
When $p=2$ or the central character of $\pi$ is highly ramified, we have the same statement, except when $c\geq 4$ is an even integer and $i=c/2$. In that case, we only claim $n_\nu\geq -c$.
\end{prop}

\begin{rem}
Following Remark \ref{remrelatingscandlevel}, this result is just to say that the representation $\pi$ is minimal under certain conditions.
\end{rem}
As a direct corollary, we have the following result about the Whittaker functional for supercuspidal representations:
\begin{cor}\label{Wiofsc}
\begin{enumerate}
\item $W^{(c)} (\alpha)=\charf_{1,0}$.
\item For general $0\leq i<c$, $W^{(i)} (\alpha)$ is supported only at $v(\alpha)=\min\{0, 2i-c\}$, consisting of level $c-i$ components and also level 0 components when $i=c-1$.
\item The exception happens when $p=2$ or the central character is highly ramified, and $c\geq 4$ is an even number and $i=c/2$. In that case, $W^{(c/2)}$ is supported at $v(\alpha)\geq 0$, consisting of level $c/2$ components.
\end{enumerate}
\end{cor}


Let's see how the results above can be applied to the matrix coefficient of a supercuspidal representation in general. Let
$$\Phi(g)=<\pi(g)F,F>=\int\limits_{\F^*}\pi(g)F(x)\overline{F(x)}d^*x,$$
where $F=\charf_{1,k}\in S(\F^*)$, and $\pi$ is supercuspidal of level $c$
This function is actually bi-$K_1(\varpi^{c+k})-$invariant. But we will only make use of the right $K_1(\varpi^{c+k})-$invariance now. 

By Lemma \ref{Iwasawadecomp}, to understand $\Phi(g)$, it will be enough to understand $\Phi(\zxz{a}{m}{0}{1}\zxz{1}{0}{\varpi^i}{1})$ for $0\leq i\leq c+k$.
\begin{prop}\label{supportofMCforSC}
Suppose $p\neq 2$. 
\begin{enumerate}
\item[(i)] For $c+k-1\leq i\leq c+k$, $\Phi(\zxz{a}{m}{0}{1}\zxz{1}{0}{\varpi^i}{1})$ is supported on $v(a)=0$ and $v(m)\geq -k-1$. On the support, we have
\begin{equation}\label{matrixcoefforsceq1}
\Phi(\zxz{a}{m}{0}{1}\zxz{1}{0}{\varpi^i}{1})=\begin{cases}
1,&\text{\ if\ }v(m)\geq -k \text{\ and\ }i=c+k;\\
-\frac{1}{q-1},&\text{\ if\ }v(m)=-k-1 \text{\ and\ }i=c+k;\\
-\frac{1}{q-1},&\text{\ if\ }v(m)\geq -k \text{\ and\ }i=c+k-1.\\
\end{cases}
\end{equation}
When $v(m)=-k-1$ and $i=c+k-1$, 
\begin{equation}\label{matrixcoefforsceq2}
\int\limits_{v(m)=-k-1}\Phi(\zxz{a}{m}{0}{1}\zxz{1}{0}{\varpi^{c+k-1}}{1})dm=\frac{1}{q-1}q^k.
\end{equation}
\item[(ii)] For $0\leq i<c+k-1$, $i\neq c/2+k$, $\Phi(\zxz{a}{m}{0}{1}\zxz{1}{0}{\varpi^i}{1})$ is supported on $v(a)=\min\{0,2i-c-2k\}$, $v(m)=i-c-2k$. It is of level $c+k-i$ as a function in $a$.
\item[(iii)]When $i=c/2+k$, the conclusion in (ii) still holds except  when $p=2$ or the central character is highly ramified, and $c\geq 4$ is an even number. In that case, one can say $\Phi(\zxz{a}{m}{0}{1}\zxz{1}{0}{\varpi^i}{1})$ is supported on $v(a)\geq 0$, $v(m)=i-c-2k=-c/2-k$. It is of level $c/2$ in $a$.
\end{enumerate}
\end{prop}

\begin{proof}
By definition,
\begin{equation}
 \Phi(g)=\int\limits_{v(x)=k}\pi(g)F(x)d^*x.
\end{equation}

To get a non-zero value for $\Phi(g)$, we just need a level 0 component supported at $v(x)=k$ for $\pi(g)F(x)$. 
We first assume that $p\neq 2$, the central character is of level $\leq 1$, and $0\leq i<c+k-1$.
According to Proposition \ref{Propofnmu}, 
$$\pi(\zxz{1}{0}{\varpi^i}{1})\charf_{1,k}(x)=\pi(-\omega\zxz{1}{-\varpi^i}{0}{1}\omega)\charf_{1,k}(x)$$
is supported at $v(x)=\min\{k,2i-c-k\}$, being a linear combination of all level $c+k-i$ characters.
By definition,
\begin{equation}
 \pi(\zxz{a}{m}{0}{1}\zxz{1}{0}{\varpi^i}{1})\charf_{1,k}(x)=\psi(mx)\pi(\zxz{1}{0}{\varpi^i}{1})\charf_{1,k}(ax),
\end{equation}
\begin{equation}
\Phi(\zxz{a}{m}{0}{1}\zxz{1}{0}{\varpi^i}{1})=\int\limits_{v(x)=k}\psi(mx)\pi(\zxz{1}{0}{\varpi^i}{1})\charf_{1,k}(ax)d^*x.
\end{equation}
One can see that we need $$v(a)=\min\{0,2i-c-2k\}$$ to change the support of $\pi(\zxz{1}{0}{\varpi^i}{1})\charf_{1,k}(ax)$ to $v(x)=k$.

When $0\leq i<c+k-1$, we need $\psi(mx)$ also to be of level $c+k-i$ at $v(x)=k$ to get level 0 components from the product. So it's supported at $$v(m)=i-c-2k.$$ It's clear now that $\Phi(\zxz{a}{m}{0}{1}\zxz{1}{0}{\varpi^i}{1})$ as a function of $a$ or $m$ is of level $c+k-i$.
So (ii) is proved. 

(iii) can be proved using the same method.

When one use the same method for (i), there will be two differences which are worth noting. 
The first difference is that when $i=c+k-1$, $\pi(\zxz{1}{0}{\varpi^i}{1})\charf_{1,k}(x)$ is a linear combination of level 1 and also level 0 components.
The second difference is that $\psi(mx)$ has level 0 component at $v(x)=k$ when $v(m)\geq -k-1$. 

Now we will prove (\ref{matrixcoefforsceq2}) and leave (\ref{matrixcoefforsceq1}) to the readers, as the latter is actually much easier to check.

So suppose $i=c+k-1$, $v(a)=0$ and $v(m)=-k-1$. Then $\pi(\zxz{1}{0}{\varpi^{c+k-1}}{1})\charf_{1,k}(ax)$ and $\psi(mx)$ will both be linear combinations of level 1 and level 0 characters.

\begin{align*}
 \int\limits_{v(m)=-k-1}\Phi(\zxz{a}{m}{0}{1}\zxz{1}{0}{\varpi^{c+k-1}}{1})dm&=\int\limits_{v(m)=-k-1}\int\limits_{v(x)=k}\psi(mx)\pi(\zxz{1}{0}{\varpi^{c+k-1}}{1})\charf_{1,k}(ax)d^*xdm\\
&=\int\limits_{v(x)=k}\int\limits_{v(m)=-k-1}\psi(mx)\pi(\zxz{1}{0}{\varpi^{c+k-1}}{1})\charf_{1,k}(ax)dmd^*x\\
&=-q^k\int\limits_{v(x)=k}\pi(\zxz{1}{0}{\varpi^{c+k-1}}{1})\charf_{1,k}(ax)d^*x
 \end{align*}
The last step is to see that the level 0 component of $\pi(\zxz{1}{0}{\varpi^{c+k-1}}{1})\charf_{1,k}$ is $-\frac{1}{q-1}\charf_{1,k}$.
\end{proof}


\section{Upper bound for the global period integral}
From now on we take $G=\D^*$ as decided in Theorem \ref{thmofJacquetconj}. Denote $X=Z_{\A}\D^*(\F)\backslash \D^*(\A)$.
Let $\pi_i, i=1,2,3$ be three unitary automorphic cuspidal representations of $\GL_2$. Let $\pi_i^{\D}$ be the image of $\pi_i$ under Jacquet-Langlands correspondence. They are naturally embedded in $L^2(X)$. We will fix $\pi_1$ and $\pi_2$ and let $\pi_3$ have varying finite conductor, but with bounded components at infinity.
\begin{defn}\label{defofconductor}
 At a local place $v$, let $c_i$ denote the levels of $\pi_i$ at $v$.
 Let 
 $$S=\{v|c_3\geq 2\max\{c_1,c_2\} \text{\ at\ }v\}.$$ 
 Let $\mathcal{N}=\prod_{v\in S}\varpi_v^{c_3-c_2}$
 and  $N=\text{Nm}(\mathcal{N})=\prod_v |\varpi_v|^{-(c_3-c_2)}$.
\end{defn}
\begin{rem}\label{remofdifferentN}
 Note that we don't take $\mathcal{N}$ here to be exactly the conductor of $\pi_3$. But their difference is controlled by the conductors of $\pi_1$ and $\pi_2$ which are fixed. In particular this difference is negligible when we consider the asymptotic behavior. 
\end{rem}
We claim here without proof that for $v\in S$, the local epsilon factor $\epsilon_v(\Pi_v,1/2)=1$, so $\D$ is the matrix algebra at these places. (We will prove this claim in Corollary \ref{coroflocalepsiloninS}. ) For this reason, the following definition makes sense:

\begin{defn}
 For $\mathcal{N}$ defined as above, let 
 
 $$a_v([\mathcal{N}])=\zxz{\varpi_v^{-(c_3-c_2)}}{0}{0}{1},$$
 and
  $$a([\mathcal{N}])=\prod_{v}a_v([\mathcal{N}]).$$
$a([\mathcal{N}])$ can be naturally embedded into $Z_\A \D^*(\F)\backslash \D^*(\A)$.

\end{defn}
Take cusp forms $f_i\in\pi_i^\D,i=1,2,3$ . We want to bound the global period integral
\begin{equation}
\I(f_1,\rho(a([\mathcal{N}]))f_2,f_3)= \int_X f_1(x)f_2(xa([\mathcal{N}]))f_3(x)dx.
\end{equation}
But before that, let's specify a little more about our choices of local components for $f_i$'s. 
\begin{enumerate}
 \item[(i)] At almost all places when all three representations are unramified, we will just choose local components to be spherical;
 \item[(ii)] For places in $S$, we will always pick newforms for all local components;
 \item[(iii)] For the remaining places, we will pick proper newforms or old forms to guarantee that the local integral $I_v\geq \delta$ for some $\delta>0$. In particular the local component of
 $f_1$ and $f_2$ can be chosen from a finite set of test vectors.
 \item[(iv)] $f_i$'s are globally and locally normalized.
\end{enumerate}
\begin{rem}\label{remofplacesoutsideS}
(iii) is guaranteed because the level of $\pi_3$ is controlled by the levels of $\pi_1$ and $\pi_2$ for places outside $S$. It's essentially proven in Lemma 6.3 and Lemma 6.4 of \cite{MW12}. Basically if we fix the level of $\pi_3$, the parametrization of all possible representations with fixed central character is compact. 
Theorem \ref{thmofJacquetconj} will guarantee that the local integral $I_v$ is not zero with a proper choice of test vectors, then $I_v$ can be bounded away from zero in an open neighborhood of the parametrization. Then (iii) is true because of compactness.

\end{rem}

Now we can state our result on the upper bound of global period integral:
\begin{prop}\label{propofglobalupbd}
 Let $\pi_i,i=1,2,3$ be three unitary automorphic cuspidal representations with $\pi_1$ and $\pi_2$ fixed. Let $f_i\in\pi_i,i=1,2$ and $\varphi\in\pi_3$ be cusp forms with local components specified as above. Then 
 \begin{equation}
 \I(f_1,\rho(a([\mathcal{N}]))f_2,f_3)=\int_X f_1(x)f_2(xa([\mathcal{N}]))f_3(x)dx\ll N^{-\delta},
 \end{equation}
 where $\delta$ can be taken to be any positive number less than $-\frac{(\alpha-1/2)(2\alpha-1/2)}{4\alpha-3}>\frac{1}{24}$ for $\alpha=7/64$.

\end{prop}
\begin{proof}
We will basically follow the proof as in \cite{AV10}.
First we specify a signed measure $\sigma$ on $G(\A_\F)$ we are going to use. We will take $\sigma=\sum_{\mathfrak{n}}a_\mathfrak{n}\mu_\mathfrak{n}$, where $\mu_\mathfrak{n}$ is the measure associated to $\mathfrak{n}-$th Hecke operator as defined in Section 2.
We will choose the sequence of complex numbers $a_\mathfrak{n}$ as follows:

Let $b$ be a fixed small positive real number to be chosen. For every finite place $v$, let $\mathfrak{l}$ be a maximal prime ideal there.  Let $T$ be the set of places where $\text{Nm}(\mathfrak{l})\in [N^b,2N^b]$ and $\pi_i$'s are unramified.
In particular by the choice of local component (i), $f_i$'s are spherical at these places. As $f_1$ and $f_2$ have fixed conductor, and primes involved in the conductor of $\pi_3$ are asymptotically less than $N^\epsilon$ for any $\epsilon>0$ as $N\rightarrow \infty$, $T$ will essentially contain all the primes with norm in $[N^b, 2N^b]$. More specifically by the distribution of primes, 
we have $ N^{b-\epsilon}\ll |T|\ll N^{b+\epsilon}$.

For $z\in \C$ we put $\text{sign}(z)=z/|z|$ for $z\neq 0$ and $\text{sign}(0)=1$. Put
\begin{equation}
a_\mathfrak{n}=\begin{cases}
\overline{\text{sign}(\check{\lambda_\mathfrak{n}}(\mathfrak{n}))}, &\mathfrak{n}\in T \text{\ or\ } \mathfrak{n}=\mathfrak{l}^2, \mathfrak{l}\in T\\
0, &\text{\ else.}
\end{cases}
\end{equation}
Here $\check{\lambda_\mathfrak{n}}$ is the eigenvalue of the Hecke operator $\check{\mu_\mathfrak{n}}$ acting on $f_3$, as the local component of $f_3$ at this place is spherical. Then by the definition above, one can easily verify the following inequalities, which we will make use of later:
\begin{equation}\label{globalineq1}
|\sum_\mathfrak{n}a_\mathfrak{n}\check{\lambda_\mathfrak{n}}|\gg_{\epsilon,\F}N^{b-\epsilon}.
\end{equation}
\begin{equation}\label{globalineq2}
\sum_\mathfrak{n}\text{Nm}(\mathfrak{n})^{1/2+\epsilon}|a_\mathfrak{n}|\ll_\epsilon N^{2b+\epsilon}.
\end{equation}
\begin{equation}\label{globalineq3}
\sum_{\mathfrak{n},\mathfrak{m}}\sum_{\mathfrak{d}|(\mathfrak{n},\mathfrak{m})}(\text{Nm}(\frac{\mathfrak{n}\mathfrak{m}}{\mathfrak{d}^2}))^{2\alpha-1/2}|a_\mathfrak{n}||a_\mathfrak{m}|\ll N^{(4\alpha+1)b}
\end{equation}
The first equality follows from Corollary \ref{corofboundoneigenvalues}. The second and the third inequalities are more direct to check. $\alpha$ in the last inequality is a bound towards Ramanujan conjecture, and we need the fact that one can take $\alpha<1/4$.

Now for the measure $\sigma$ defined as above, we have $f_3*\check{\sigma}=\lambda f_3$, where
\begin{equation}
 \lambda=\sum_\mathfrak{n}a_\mathfrak{n}\check{\lambda_\mathfrak{n}}.
\end{equation}
Let $\Psi(x)=f_1(x)f_2(xa([\mathcal{N}]))\in C^\infty(X)$. Then
\begin{align}\label{globalfirststep}
 \lambda\I&=\int_X \Psi(x)(f_3*\check{\sigma})(x)dx=\int_X(\Psi*\sigma)(x)f_3(x)dx\leq (\int_X|\Psi*\sigma|^2dx)^{1/2}\\
 &=(\int_X\int_{g,g'\in G(\A_\F)}(\rho(g)\Psi)\overline{(\rho(g')\Psi)}d\sigma(g)d\sigma(g')dx)^{1/2}\notag\\
 &=(\int_X\int_{g,g'\in G(\A_\F)}f_1(xg)f_2(xa([\mathcal{N}])g)\overline{f_1(xg')f_2(xa([\mathcal{N}])g')}d\sigma(g)d\sigma(g')dx)^{1/2}\notag\\
 &=(\int_X\int_{g,g'\in G(\A_\F)}f_1(xg)f_2(xga([\mathcal{N}]))\overline{f_1(xg')f_2(xg'a([\mathcal{N}]))}d\sigma(g)d\sigma(g')dx)^{1/2}\notag
 \end{align}
 
In the last equality, we have used that according to our choice of $\sigma$, the support of $\sigma$ commmutes with $a([\mathcal{N}])$. Now we want to change the order of the integral, separate the constant part and use Proposition \ref{propofboundglobalMC} to bound the difference. In particular, 
let $h_i(x)=f_i(xg)\overline{f_i(xg')}, i=1,2$, so the center acts trivially on $h_i(x)$. Then we have
\begin{align}\label{globalsecondstep}
 &\text{\ \ }|\int_Xh_1(x)h_2(xa(\mathcal{N}))dx-\sum_{\chi^2=1}\chi(\mathcal{N})\int_X h_1(x)\chi(x)dx\int_X h_2(x)\chi(x)dx|\\
 &=|<h_1,\rho(a(\mathcal{N}))h_2>-<\mathcal{P}h_1, \rho(a(\mathcal{N}))\mathcal{P}h_2>|\notag\\
 &\ll N^{\alpha-1/2+\epsilon}||h_1||_{L^2}||h_2||_{L^2} \notag\\
 &\ll N^{\alpha-1/2+\epsilon}\notag
\end{align}
The implicit constant depends on the compact open subgroups that stabilize $f_1$ and $f_2$ at places in $S$, thus is bounded. In the last inequality we have used
$ ||h_i||_{L^2}\leq ||f_i||^2_{L^4}$, which is finite and bounded because $f_i$'s are normalized cusp forms chosen from a finite fixed collection for $i=1,2$.

Combining (\ref{globalfirststep}) and (\ref{globalsecondstep}), we have
\begin{align}\label{global3step}
 |\lambda\I|^2\ll N^{\alpha-1/2+\epsilon}||\sigma||^2
 +\sum_{\chi^2=1}\int_{g,g'}|<\rho(g^{-1}g')f_1,f_1\otimes\chi><\rho(g^{-1}g')f_2,f_2\otimes\chi>|d|\sigma|(g)d|\sigma|(g'),
\end{align}
where $|\sigma|=\sum_{\mathfrak{n}}|a_\mathfrak{n}|\mu_{\mathfrak{n}}$ is the total variation measure associated to $\sigma$, $||\sigma||=|\sigma|(X)$ is the total variation of $\sigma$. 

Note that if we consider $|<\rho(g^{-1}g')f_1,f_1\otimes\chi><\rho(g^{-1}g')f_2,f_2\otimes\chi>|$ as a function of $g$ or $g'$, the center acts on it trivially as the central characters are unitary. So Hecke operators $\mu_\mathfrak{n}$ act on it nicely. In particular,
define $\sigma^{(2)}=|\sigma|*|\sigma|$. Then we can rewrite the above result as
\begin{equation}\label{global4step}
  |\lambda\I|^2\ll N^{\alpha-1/2+\epsilon}||\sigma||^2
 +\sum_{\chi^2=1}\int_{g}|<\rho(g)f_1,f_1\otimes\chi><\rho(g)f_2,f_2\otimes\chi>|d\sigma^{(2)}(g).
\end{equation}
According to Lemma \ref{convolutionofHecke}, we have the following for spherical local components on which the center acts trivially:
\begin{equation}\label{globaling1}
\sigma^{(2)}=\sum_{\mathfrak{n},\mathfrak{n}}|a_\mathfrak{n}||a_\mathfrak{m}|\sum_{\mathfrak{d}|(\mathfrak{n},\mathfrak{m})}\mu_{\mathfrak{n}\mathfrak{m}\mathfrak{d}^{-2}}.
\end{equation}


According to Lemma \ref{localboundmatrixcoeff}, one can easily prove that for spherical functions:
\begin{equation}
 \int_{g\in G(\A_\F)}|<\rho(g)f_1,f_1\otimes\chi><\rho(g)f_2,f_2\otimes\chi>|d\mu_\mathfrak{n}(g)\ll_\epsilon\text{Nm}(\mathfrak{n})^{2\alpha-1/2+\epsilon}.
\end{equation}
Moreover, for fixed $g\in \text{Supp}(\mu_\mathfrak{n})$, the inner product $<\rho(g)f_1,f_1\otimes\chi>$ is nonvanishing only if $\chi$ is unramified at all places not dividing $\mathfrak{n}$ and $f_1$ is unramified. The number of such quadratic characters is $O_\epsilon(\text{Nm}(\mathfrak{n})^\epsilon N ^\epsilon)$, where the implicit constant is allowed to depend on the base field $\F$. Thus
\begin{equation}\label{globaling2}
\sum_{\chi^2=1}\int_{g\in G(\A_\F)}|<\rho(g)f_1,f_1\otimes\chi><\rho(g)f_2,f_2\otimes\chi>|d\mu_\mathfrak{n}(g)\ll_\epsilon\text{Nm}(\mathfrak{n})^{2\alpha-1/2+\epsilon}N^{\epsilon}.
\end{equation}
One can also check that
\begin{equation}\label{globaling3}
||\sigma||\ll_\epsilon \sum_{\mathfrak{n}}\text{Nm}(\mathfrak{n})^{1/2+\epsilon}|a_\mathfrak{n}|.
\end{equation}
Now combine formulae (\ref{globaling1}), (\ref{globaling2}) and (\ref{globaling3}) into (\ref{global4step}), we have
\begin{equation}
|\I|\ll N^\epsilon\frac{((\sum_\mathfrak{n}\text{Nm}(\mathfrak{n})^{1/2+\epsilon}|a_\mathfrak{n}|)^2N^{\alpha-1/2+\epsilon}+\sum_{\mathfrak{n},\mathfrak{m}}\sum_{\mathfrak{d}|(\mathfrak{n},\mathfrak{m})}(\text{Nm}(\frac{\mathfrak{n}\mathfrak{m}}{\mathfrak{d}^2}))^{2\alpha-1/2+\epsilon}|a_\mathfrak{n}||a_\mathfrak{m}|)^{1/2}}
{|\sum_\mathfrak{n}a_\mathfrak{n}\check{\lambda_\mathfrak{n}}|}.
\end{equation}
Now we make use of the inequalities (\ref{globalineq1}), (\ref{globalineq2}) and (\ref{globalineq3}) and get
\begin{equation}
|\I|\ll N^\epsilon\frac{(N^{4b+\alpha-1/2}+N^{(4\alpha+1)b})^{1/2}}{N^b}.
\end{equation}
Now pick $b=\frac{\alpha-1/2}{4\alpha-3}>0$ as we can pick $\alpha<1/4$. Then the above inequality becomes
\begin{equation}
|\I|\ll N^{\frac{(\alpha-1/2)(2\alpha-1/2)}{4\alpha-3}+\epsilon}.
\end{equation}
Again $\frac{(\alpha-1/2)(2\alpha-1/2)}{4\alpha-3}<0$. When we pick $\alpha=7/64$,
\begin{equation}
 \frac{(\alpha-1/2)(2\alpha-1/2)}{4\alpha-3}=-\frac{225}{5248}<-\frac{1}{24}.
\end{equation}

\end{proof}

\begin{rem}\label{remofN1N2}
 
 The roles of $f_1$ and $f_2$ are interchangeable. One can also, for example, assume $\mathcal{N}=\mathcal{N}_1\mathcal{N}_2$ with $\mathcal{N}_1$ and $\mathcal{N}_2$ relatively prime, and get a similar inequality
 \begin{equation}
  \int_X f_1(xa([\mathcal{N}_1]))f_2(xa([\mathcal{N}_2]))f_3(x)dx\ll N^{-\delta}.
 \end{equation}

\end{rem}

\section{Local integral for the triple product $L$-function}

In this section, we shall compute the local integral for the triple product $L-$function explicitly. As we will work purely locally, let's suppress subscript $v$ in this section. 

Let $\pi_i, i=1,2,3$ be three local irreducible unitary representations of $\GL_2$. Let $f_i\in \pi_i$ be the normalized newforms for places $v\in S$ according to our choice in the last section. Let $\Phi_i=<\pi_i(g)f_i,f_i>$ for $i=1,3$ and $\Phi_2(g)=<\pi_2(ga_v([\mathcal{N}]))f_2,\pi_2(a_v([\mathcal{N}])) f_2>$. We will compute in this section the following integral
\begin{equation}\label{localtarget2}
\int\limits_{\F^*\backslash\GL_2(\F)}\Phi_1(g)\Phi_2(g)\Phi_3(g)dg.
\end{equation}
We will assume that $c_3\geq 2\max\{c_1,c_2,1\}$. The difference between this assumption and the condition for the set of places $S$ is the case $c_1=c_2=0$ and $c_3=1$. But this case was already considered in \cite{MW12}.
In general the exact value of the matrix coefficient is very difficult to write out explicitly, and so is the local integral (\ref{localtarget2}). But with the assumption $c_3\geq 2\max\{c_1,c_2,1\}$, the computations turn out to be very nice and simple. 

We will consider all possible local irreducible unitary representations which fall into the following three types:
\begin{enumerate}
\item[Type 1.]$\pi$ supercuspidal or of form $\pi(\mu_1,\mu_2)$ where $\mu_i$ is ramified of level $k_i>0$ for $i=1,2$;
\item[Type 2.]$\pi$ unramified or special unramified;
\item[Type 3.]$\pi$ of form $\pi(\mu_1,\mu_2)$ where $\mu_1$ is unramified and $\mu_2$ ramified of level $k$.
\end{enumerate}
Note we don't have to consider the case when $\mu_1$ is ramified and $\mu_2$ is unramified as $\pi(\mu_1,\mu_2)\cong\pi(\mu_2,\mu_1)$. Also when $\pi$ is of form $\sigma(\chi|\cdot|^{1/2},\chi|\cdot|^{-1/2})$ where $\chi$ is ramified, we can pick the newform similarly as in the second case of Type 1. So this case won't be considered as a different case.

%

\begin{theo}\label{thmlocalint}
Let $\pi_1$, $\pi_2$, $\pi_3$ be three local irreducible unitary  representations of $\GL_2$, with levels satisfying $c_3\geq 2\max\{c_1,c_2,1\}$. Then the local integral
\begin{equation}
\int\limits_{\F^*\backslash\GL_2(\F)}\Phi_1(g)\Phi_2(g)\Phi_3(g)dg=\frac{(1-A)(1-B)}{(q+1)q^{c_3-1}},
\end{equation}
where 
$$A=\Phi_1(\zxz{1}{\varpi^{-1}}{0}{1})\text{\ \ and\ \ }B=\Phi_2(\zxz{1}{0}{\varpi^{c_3-1}}{1}).$$
More specifically we have the following tables of values of $A$ and $B$ for all three types of irreducible unitary representations
 
\begin{tabular}{|p{1cm}|p{2.5cm}|p{4.8cm}|p{3.5cm}|p{2.5cm}|}
\hline
$\pi_1$	&Type 1 &unramified of form $\pi(\chi_1,\chi_2)$	&special unramified	&Type 3\\ \hline
A	&$-\frac{1}{q-1} $	&$\frac{1}{q+1}(\frac{\chi_1}{\chi_2}+\frac{\chi_2}{\chi_1}+1-q^{-1})$	&$-q^{-1}$	&0 \\\hline
\end{tabular}

\begin{tabular}{|p{1cm}|p{2.5cm}|p{4.8cm}|p{3.5cm}|p{2.5cm}|}
\hline
$\pi_2$	&Type 1 &unramified of form $\pi(\eta_1,\eta_2)$	&special unramified	&Type 3\\ \hline
B	&$-\frac{1}{q-1} $	&$\frac{1}{q+1}(\frac{\eta_1}{\eta_2}+\frac{\eta_2}{\eta_1}+1-q^{-1})$	&$-q^{-1}$	&0 \\\hline
\end{tabular}

\end{theo}
\vspace{0.5cm}

\subsection{General strategy}

As all the matrix coefficients will be right $K_1(\varpi^{c_3})-$invariant, it's natural to separate the integral on $\F^*\backslash\GL_2(\F)$ into integrals on the sets of the form
$$\zxz{a}{m}{0}{1}\zxz{1}{0}{\varpi^i}{1}K_1(\varpi^{c_3})$$
for $0\leq i\leq c_3$. So one would like to know the values of $\Phi_i$ on matrices of the form $\zxz{a}{m}{0}{1}\zxz{1}{0}{\varpi^i}{1}$.

As we have assumed that $c_3\geq 2\max\{c_1,c_2,1\}$, $\pi_3$ will always be of Type 1 where $k_1=k_2$ if $\pi_3$ is induced from $\mu_1$ and $\mu_2$. This is because we have assumed that the product of central characters is always trivial.

We shall first figure out the properties of $\Phi_3$:

\begin{lem}\label{lemofPhi3}
Let $\pi_3$ be a supercuspidal representation or $\pi(\mu_1,\mu_2)$ where $\mu_1$ $\mu_2$ are both of level $c_3/2$. 
\begin{enumerate}
\item[(1)]When $i=c_3$ or $c_3-1$, $\Phi_3(\zxz{a}{m}{0}{1}\zxz{1}{0}{\varpi^i}{1})$ is supported at $v(a)=0$ and $v(m)\geq -1$. We have the following special values on the support:
\begin{equation}
\Phi_3(\zxz{a}{m}{0}{1}\zxz{1}{0}{\varpi^i}{1})=\begin{cases}
1,&\text{\ if\ }v(m)\geq 0 \text{\ and\ }i=c_3;\\
-\frac{1}{q-1},&\text{\ if\ }v(m)=-1 \text{\ and\ }i=c_3;\\
-\frac{1}{q-1},&\text{\ if\ }v(m)\geq 0 \text{\ and\ }i=c_3-1.\\
\end{cases}
\end{equation}
When $v(a)=0$, $v(m)=-1$ and $i=c_3-1$, 
\begin{equation}\label{Phi3ramint}
\int\limits_{v(m)=-1}\Phi_3(\zxz{a}{m}{0}{1}\zxz{1}{0}{\varpi^{c_3-1}}{1})dm=\frac{1}{q-1}.
\end{equation}

\item[(2)]When $i\geq c_3/2$, $\Phi_3(\zxz{a}{m}{0}{1}\zxz{1}{0}{\varpi^i}{1})$ is supported at $v(m)\geq -c_3/2$ and consists of level$\leq c_3/2$ characters in $a$. It doesn't contain level 0 components in $a$ unless $i=c_3$ or $i=c_3-1$.
\item[(3)]When $0\leq i<c_3/2$, $\Phi_3(\zxz{a}{m}{0}{1}\zxz{1}{0}{\varpi^i}{1})$ is supported at $v(a)=2i-c_3$ and $v(m)=i-c_3$. As a function in $a$, it consists of level $c_3-i$ characters.
\end{enumerate}
\end{lem}
\begin{proof}
When $\pi_3$ is supercuspidal, the above results follow directly from (actually is weaker than) Proposition \ref{supportofMCforSC}. When $\pi_3$ is of the form $\pi(\mu_1, \mu_2)$ where $\mu_i$ are both of level $c_3/2$, the claims basically follow from Lemma \ref{Wiofram}. By definition
\begin{align}\label{Phi3formula}
\Phi_3(\zxz{a}{m}{0}{1}\zxz{1}{0}{\varpi^{i}}{1})&=\int\psi(m\alpha)W^{(i)}(a\alpha)\overline{W^{(c_3)}(\alpha)}d^*\alpha\\
&=\int\limits_{v(\alpha)=0}\psi(m\alpha)W^{(i)}(a\alpha)d^*\alpha.\notag
\end{align}
The special values and the special integral just follow from (ii) of Lemma \ref{Wiofram}. As an example, we will prove (\ref{Phi3ramint}). When $v(a)=0$ and $i=c_3-1$,
\begin{align}
&\text{\ \ \ }\int\limits_{v(m)=-1}\Phi_3(\zxz{a}{m}{0}{1}\zxz{1}{0}{\varpi^{c_3-1}}{1})dm=\int\limits_{v(\alpha)=0}\int\limits_{v(m)=-1}\psi(m\alpha)W^{(c_3-1)}(a\alpha)dm d^*\alpha \\
&=-\int\limits_{v(\alpha)=0}W^{(c_3-1)}(a\alpha) d^*\alpha=-\int\limits_{v(\alpha)=0}W^{(c_3-1)}(\alpha) d^*\alpha=\frac{1}{q-1}.\notag
\end{align}
\vspace{0.5cm}

Now to prove (2), suppose $v(m)<-c_3/2$ in (\ref{Phi3formula}). Then $\psi(m\alpha)$ is of level $\geq c_3/2+1$ in $\alpha$. But one can check explicitly from (ii) and (iii) of Lemma \ref{Wiofram} that $W^{(i)}$ is of level $\leq c_3/2$ in $\alpha$.
For example, when $i=c_3/2$,
\begin{align}
&\text{\ \ \ } W^{(\frac{c_3}{2})}(a\alpha)\\
&=C^{-1}\int\limits_{v(u)\leq -\frac{c_3}{2},u\notin \varpi^{-\frac{c_3}{2}}(-1+\varpi O_F)}\mu_1(-\frac{\varpi^{\frac{c_3}{2}}}{1+u\varpi^{\frac{c_3}{2}}})\mu_2(-a\alpha u)|\frac{\varpi^{\frac{c_3}{2}}}{a\alpha u(1+u\varpi^{\frac{c_3}{2}})}|^{1/2}\psi(-a\alpha u)q^{-v(a\alpha)}du.\notag
\end{align}
As functions in $u$, $\mu_1(-\frac{\varpi^{\frac{c_3}{2}}}{1+u\varpi^{\frac{c_3}{2}}})\mu_2(-a\alpha u)$ is of level $\leq c_3/2$, $\psi(a\alpha u)$ if of level $-v(a\alpha u)$. Then $v(a\alpha u)\geq -c_3/2$ for the integral in $u$ to be nonzero. Then the level of $W^{(i)}(a\alpha)$ in $\alpha$ (and also in $a$) is $\leq c_3/2$. Then (\ref{Phi3formula}) has to be zero as it's
 the integral of product of level$\geq c_3/2+1$ components with level$\leq c_3/2$ components. One can also see from this argument that the level of $\Phi_3$ in $a$ is $\leq c_3/2$.
 
To find the level 0 component of $\Phi_3(\zxz{a}{m}{0}{1}\zxz{1}{0}{\varpi^{i}}{1})$ in $a$ is equivalent to find the level 0 component in $W^{(i)}$, which only occurs when $i=c_3$ or $i=c_3-1$ from Lemma \ref{Wiofram}.

\vspace{0.5cm}
(3) follows from (i) of Lemma \ref{Wiofram}: when $i<c_3/2$,
\begin{equation}
W^{(i)}(a\alpha)=C^{-1}\int\limits_{u\in O_F^*}\mu_1(-\frac{\varpi^i}{u})\mu_2(a\alpha\varpi^{-i}(1-\varpi^{\frac{c_3}{2}-i}u))q^{v(a\alpha)/2-i}\psi(a\alpha\varpi^{-i}(1-\varpi^{\frac{c_3}{2}-i}u))q^{2i-\frac{c_3}{2}-v(a\alpha)}du.
\end{equation}
As functions in $u$, $\mu_1(-\frac{\varpi^i}{u})$ is multiplicative of level $c_3/2$, $\mu_2(a\alpha\varpi^{-i}(1-\varpi^{\frac{c_3}{2}-i}u))$ is of level $i<c_3/2$, $\psi(a\alpha\varpi^{-i}(1-\varpi^{\frac{c_3}{2}-i}u))$ is of level $2i-c_3/2-v(a\alpha)$. As $v(\alpha)=0$ in (\ref{Phi3formula}), then $W^{(i)}(a\alpha)$ is not zero only when $v(a)=2i-c_3$. We will assume this for the remaining discussions.

As functions in $\alpha$, $\mu_2(a\alpha\varpi^{-i}(1-\varpi^{\frac{c_3}{2}-i}u))$ is of level $c_3/2$, and $\psi(a\alpha\varpi^{-i}(1-\varpi^{\frac{c_3}{2}-i}u))$ is of level $i-v(a\alpha)=c_3-i>c_3/2$. Then $W^{(i)}(a\alpha)$ as a function in $\alpha$ is of level $c_3-i$. Thus for the integral in (\ref{Phi3formula}) to be vanishing, we require $v(m)=i-c_3$. 

In this argument, one can easily see that $\Phi_3(\zxz{a}{m}{0}{1}\zxz{1}{0}{\varpi^{i}}{1})$ as a function in $a$ for $i<c_3/2$ is of level $c_3-i$.
\end{proof}

Now we can explain the strategy to prove Theorem \ref{thmlocalint}. As we mentioned earlier, we will add up the integrals on the double cosets of the form
$$\zxz{a}{m}{0}{1}\zxz{1}{0}{\varpi^i}{1}K_1(\varpi^{c_3})$$
for $0\leq i\leq c_3$. We will show that the nonzero contribution will only come from $i=c_3$ and $i=c_3-1$, where we know special values or integrals for $\Phi_3$. In particular, we will prove the following two claims about $\Phi_1$ and $\Phi_2$ for various types of representations:

\begin{claim1}
\begin{enumerate}
\item[(1)]When $i\geq c_3/2$, $\Phi_1(\zxz{a}{m}{0}{1}\zxz{1}{0}{\varpi^i}{1})$ is of level 0 in $a$ for fixed valuations. In particular we have the following special values:
$$\Phi_1(\zxz{a}{m}{0}{1}\zxz{1}{0}{\varpi^i}{1})=\begin{cases}
1, &\text{\ if\ }v(a)=0 \text{\ and\ }v(m)\geq 0\\
A, &\text{\ if\ }v(a)=0 \text{\ and\ }v(m)=-1.
\end{cases}$$
\item[(2)]When $i<c_3/2$, $v(a)=2i-c_3$ and $v(m)=i-c_3$, $\Phi_1(\zxz{a}{m}{0}{1}\zxz{1}{0}{\varpi^i}{1})$ as a function in $a$ is of level $\leq c_1<c_3-i$.
\end{enumerate}
\end{claim1}

\begin{rem}
This claim should be clear by intuition. When $i\geq c_3/2\geq c_1$, $\Phi_1(\zxz{a}{m}{0}{1}\zxz{1}{0}{\varpi^i}{1})=\Phi_1(\zxz{a}{m}{0}{1})$. Its special value at $v(a)=0$ and $v(m)\geq 0$ is just a matter of normalization.
\end{rem}

\begin{claim2}
\begin{enumerate}
\item[(1)]For $i\geq c_3/2$, and $v(m)\geq -c_3/2$, $\Phi_2(\zxz{a}{m}{0}{1}\zxz{1}{0}{\varpi^i}{1})$ is 
of level 0 as a function in $a$ and independent of $m$. When $i=c_3$, $v(a)=0$ and $v(m)\geq -c_3/2$, 
\begin{equation}
\Phi_2(\zxz{a}{m}{0}{1})=1.
\end{equation}
When $i=c_3-1$, $v(a)=0$ and $v(m)\geq -c_3/2$, 
\begin{equation}
\Phi_2(\zxz{a}{m}{0}{1}\zxz{1}{0}{\varpi^{c_3-1}}{1})=B.
\end{equation}

\item[(2)]For $i<c_3/2$, $v(a)=2i-c_3$ and $v(m)=i-c_3$, $\Phi_2(\zxz{a}{m}{0}{1}\zxz{1}{0}{\varpi^i}{1})$ is 
of level $\leq c_2<c_3-i$ as a function in $a$. 
\end{enumerate}
\end{claim2}
\begin{rem}
 Again the special value when $i=c_3$, $v(a)=0$ and $v(m)\geq -c_3/2$ is just a matter of normalization.
\end{rem}

Now suppose that these two claims are always true for all three types of representations. When $i<c_3-1$, the level in $a$ of $\Phi_1$ and $\Phi_2$ is strictly less than that of $\Phi_3$.  So indeed the only nonzero contribution to the final integral will come from $i=c_3$ and $i=c_3-1$. One can then have the following tables of values:

\vspace{0.5cm}

\begin{tabular}{|p{3cm}|p{3cm}|p{3cm}|p{3cm}|}
\hline
$v(a)$ always 0	&$\Phi_1$	&$\Phi_2$	&$\Phi_3$\\
\hline
$i=c_3,v(m)\geq 0$	&1	&1	&1\\
\hline
$i=c_3,v(m)=-1$	&A	&1	&$-\frac{1}{q-1}$\\
\hline
$i=c_3-1,v(m)\geq 0$	&1	&B	&$-\frac{1}{q-1}$\\
\hline
$i=c_3-1,v(m)=-1$	&A	&B	&satisfying (\ref{Phi3ramint})\\
\hline
\end{tabular}

\vspace{0.5cm}

Then by Lemma \ref{localintcoefficient}, one can easily compute that
\begin{align}
&\int\limits_{\F^*\backslash\GL_2(\F)}\Phi_1(g)\Phi_2(g)\Phi_3(g)dg\\
=&\frac{1}{(q+1)q^{c_3-1}}[1+(q-1)A(-\frac{1}{q-1})]+\frac{q-1}{(q+1)q^{c_3-1}}[B(-\frac{1}{q-1})+AB\frac{1}{q-1}]\notag\\
=&\frac{(1-A)(1-B)}{(q+1)q^{c_3-1}}.\notag
\end{align}
So the theorem will be proved if we can verify Claim 1 and Claim 2 for various types of representations. We will do this in the remaining of this section. Before that, let's give the formulae for $\Phi_1$ and $\Phi_2$ more explicitly in Whittaker functionals. Let $W_i$ be the corresponding Whittaker functionals for the normalized newforms $f_i$.
$\Phi_1$ is right $K_1(\varpi^{c_1})-$invariant, and thus automatically $K_1(\varpi^{c_3})-$invariant. When $i\geq c_1$,
\begin{equation}\label{Phi1_1}
\Phi_1(\zxz{a}{m}{0}{1}\zxz{1}{0}{\varpi^i}{1})=\int\limits_{\alpha}\psi(m\alpha) W_1^{(c_1)}(a\alpha)\overline{W_1^{(c_1)}(\alpha)}d^*\alpha.
\end{equation}
When $i<c_1$,
\begin{equation}\label{Phi1_2}
\Phi_1(\zxz{a}{m}{0}{1}\zxz{1}{0}{\varpi^i}{1})=\int\limits_{\alpha}\psi(m\alpha) W_1^{(i)}(a\alpha)\overline{W_1^{(c_1)}(\alpha)}d^*\alpha.
\end{equation}

Now for $\Phi_2$, by definition , we have
\begin{equation}
\Phi_2(g)=<\pi_2(g\zxz{\varpi^{-c_3+c_2}}{0}{0}{1})f_2,\pi_2(\zxz{\varpi^{-c_3+c_2}}{0}{0}{1})f_2>=<\pi_2(\zxz{\varpi^{c_3-c_2}}{0}{0}{1}g\zxz{\varpi^{-c_3+c_2}}{0}{0}{1})f_2,f_2>.
\end{equation}

It is again right $K_1(\varpi^{c_3})-$invariant as 

$$\zxz{\varpi^{c_3-c_2}}{0}{0}{1}K_1(\varpi^{c_3})\zxz{\varpi^{-c_3+c_2}}{0}{0}{1}\subset K_1(\varpi^{c_2}),$$

and $f_2$ is $K_1(\varpi^{c_2})-$invariant. 

Now if $i\geq c_3-c_2$,
\begin{equation*}
 \zxz{\varpi^{c_3-c_2}}{0}{0}{1}\zxz{a}{m}{0}{1}\zxz{1}{0}{\varpi^i}{1}\zxz{\varpi^{-c_3+c_2}}{0}{0}{1}=\zxz{a}{m\varpi^{c_3-c_2}}{0}{1}\zxz{1}{0}{\varpi^{i-c_3+c_2}}{1},
\end{equation*}
and
\begin{equation}\label{Phi2_1}
 \Phi_2(\zxz{a}{m}{0}{1}\zxz{1}{0}{\varpi^i}{1})=
 \int  \psi(m\varpi^{c_3-c_2}\alpha )W_2^{(i-c_3+c_2)}(a\alpha)\overline{W_2^{(c_2)}(\alpha)} d^*\alpha.
\end{equation}

If $i<c_3-c_2$,
\begin{align}
&\text{\ \ \ }\zxz{\varpi^{c_3-c_2}}{0}{0}{1}\zxz{a}{m}{0}{1}\zxz{1}{0}{\varpi^i}{1}\zxz{\varpi^{-c_3+c_2}}{0}{0}{1}\notag\\
&=\varpi^{i-c_3+c_2}\zxz{a\varpi^{-2i+2c_3-2c_2}}{a(\varpi^{-i+c_3-c_2}-\varpi^{-2i+2c_3-2c_2})+m\varpi^{c_3-c_2}}{0}{1}\zxz{1}{0}{1}{1}\zxz{1}{-1+\varpi^{-i+c_3-c_2}}{0}{1},\notag
\end{align}
and
\begin{align}\label{Phi2_2}
&\Phi_2(\zxz{a}{m}{0}{1}\zxz{1}{0}{\varpi^i}{1})\\
=&w_{\pi_2}(\varpi^{i-c_3+c_2})\int\psi((a(\varpi^{-i+c_3-c_2}-\varpi^{-2i+2c_3-2c_2})+m\varpi^{c_3-c_2})\alpha)W_2^{(0)}(a\varpi^{-2i+2c_3-2c_2}\alpha)\overline{W_2^{(c_2)}(\alpha)}d^*\alpha.\notag
\end{align}

Now we can actually reduce many parts of Claim 1 and Claim 2 to the following simple lemma:
\begin{lem}
 Let $\pi$ be a unitary local representation of $\GL_2$. Let $W$ be a Whittaker functional associated to a newform in $\pi$. Then 
 $W(\zxz{\alpha}{0}{0}{1})$ is of level 0 in $\alpha$ and supported on $v(\alpha)\geq 0$.
\end{lem}
\begin{proof}
 The claim follows directly from Lemma \ref{Wiofram}, Lemma \ref{Wiofbiasram} and Corollary \ref{Wiofsc} for Type 1 and Type 3. It's also well know for unramified representations. For special unramified representations $\pi\cong \sigma(\mu|\cdot|^{1/2},\mu|\cdot|^{-1/2})$ where $\mu$ is unramified, one can see, for example, \cite{YH13}. There I proved that the Whittaker functional
 associated to a newform satisfies the following formula:
$$ W(\zxz{\alpha}{0}{0}{1})=\begin{cases}
   \mu(\alpha)q^{-v(\alpha)}, &\text{\ if \ }v(\alpha)\geq 0;\\
                            0,&\text{\ otherwise}.
                            \end{cases}$$

\end{proof}
Now let's prove parts of Claim 1 and Claim 2 without computing $A$ and $B$ explicitly. First for part (1) of Claim 1, let $i\geq c_1$. Then
\begin{equation}
\Phi_1(\zxz{a}{m}{0}{1}\zxz{1}{0}{\varpi^i}{1})=\int\limits_{\alpha}\psi(m\alpha) W_1^{(c_1)}(a\alpha)\overline{W_1^{(c_1)}(\alpha)}d^*\alpha.
\end{equation}
Both of $W_1^{(c_1)}(a\alpha)$ and $\overline{W_1^{(c_1)}(\alpha)}$ are of level 0 in $\alpha$ because of the lemma above. Then the result should be of level 0 in $a$. One can make a change of variable to see $A$ is well-defined.

Now we consider part (1) of Claim 2. If $v(m)\geq -c_3/2$, then $v(m\varpi^{c_3-c_2}\alpha)\geq 0$ for $v(\alpha)\geq 0$. When $i\geq c_3-c_2$,
\begin{align}\label{Phi2_1_1}
 \Phi_2(\zxz{a}{m}{0}{1}\zxz{1}{0}{\varpi^i}{1})=&
 \int  \psi(m\varpi^{c_3-c_2}\alpha )W_2^{(i-c_3+c_2)}(a\alpha)\overline{W_2^{(c_2)}(\alpha)} d^*\alpha\\
 &=\int W_2^{(i-c_3+c_2)}(a\alpha)\overline{W_2^{(c_2)}(\alpha)} d^*\alpha.\notag
\end{align}
This is to find level 0 components of $W_2^{(i-c_3+c_2)}$, and should be of level 0 in $a$. It's also clearly independent of $m$. That's why $B$ is well-defined.

When $c_3/2\leq i<c_3-c_2$, 
\begin{align}
&\Phi_2(\zxz{a}{m}{0}{1}\zxz{1}{0}{\varpi^i}{1})\\
=&w_{\pi_2}(\varpi^{i-c_3+c_2})\int\psi((a(\varpi^{-i+c_3-c_2}-\varpi^{-2i+2c_3-2c_2})+m\varpi^{c_3-c_2})\alpha)W_2^{(0)}(a\varpi^{-2i+2c_3-2c_2}\alpha)\overline{W_2^{(c_2)}(\alpha)}d^*\alpha\notag\\
=&w_{\pi_2}(\varpi^{i-c_3+c_2})\int\psi(a(\varpi^{-i+c_3-c_2}-\varpi^{-2i+2c_3-2c_2})\alpha)W_2^{(0)}(a\varpi^{-2i+2c_3-2c_2}\alpha)\overline{W_2^{(c_2)}(\alpha)}d^*\alpha.\notag\\
\end{align}
Again this integral is to find level 0 components of $\psi(a(\varpi^{-i+c_3-c_2}-\varpi^{-2i+2c_3-2c_2})\alpha)W_2^{(0)}(a\varpi^{-2i+2c_3-2c_2}\alpha)$ and should be of level 0 in $a$ and independent of $m$.

One can also prove (2) of Claim 2 without referring to any specifc type of representation. When $i<c_3/2$, $v(a)=2i-c_3$ and $v(m)=i-c_3$, we know
\begin{align}\label{Phi2_2_1}
\Phi_2(\zxz{a}{m}{0}{1}\zxz{1}{0}{\varpi^i}{1})=w_{\pi_2}(\varpi^{i-c_3+c_2})\int\limits_{v(\alpha)\geq 0}&\psi(m\varpi^{c_3-c_2}\alpha)\psi((\varpi^{-i+c_3-c_2}-\varpi^{-2i+2c_3-2c_2})a\alpha)\\
&W_2^{(0)}(\varpi^{-2i+2c_3-2c_2}a\alpha)\overline{W^{(c_2)}(\alpha)}d^*\alpha.\notag
\end{align}
By the previous lemma, $\overline{W^{(c_2)}(\alpha)}$ is of level 0 in $\alpha$, and $v(\alpha)\geq 0$ in the integral. So $v(m\varpi^{c_3-c_2}\alpha)\geq i-c_3+c_3-c_2=i-c_2\geq -c_2$, which means $\psi(m\varpi^{c_3-c_2}\alpha)$ is of level$\leq c_2$ in $\alpha$.
Now note that if $\psi((\varpi^{-i+c_3-c_2}-\varpi^{-2i+2c_3-2c_2})a\alpha)W_2^{(0)}(\varpi^{-2i+2c_3-2c_2}a\alpha)$ has a component of level $j$ in $\alpha$, then this component is also of level $j$ in $a$. As only those component of level$\leq c_2$ will
be detected by the integral, $\Phi_2(\zxz{a}{m}{0}{1}\zxz{1}{0}{\varpi^i}{1})$ can only be of level$\leq c_2$ in $a$. Another way to argue this is just to do a change of variable for the integral.

For these reasons, we will only need to compute $A$ and $B$ explicitly and verify part (2) of Claim 1 for various types of unitary representations.

\subsection{Type 1 occuring}

First let's consider supercuspidal representations. When $\pi_1$ is supercuspidal, part (2) of Claim 1 follows directly from Proposition \ref{supportofMCforSC} for $k=0$ there and one can also easily see that $A=-\frac{1}{q-1}$.

When $\pi_2$ is supercuspidal, take $k=c_3-c_2$ in Proposition \ref{supportofMCforSC}. Note $\pi_2(\zxz{\varpi^{-c_3+c_2}}{0}{0}{1})\charf_{1,0}=\charf_{1,c_3-c_2}$, and thus
$$\Phi_2(g)=<\pi_2(g)\charf_{1,c_3-c_2},\charf_{1,c_3-c_2}>.$$
So again we can apply Proposition \ref{supportofMCforSC} and get $B=-\frac{1}{q-1}$. 



Now suppose $\pi_1$ is of form $\pi(\chi_1,\chi_2)$, where $\chi_1$ and $\chi_2$ are both ramified. For any $m$ with $v(m)=-1$,
\begin{equation}
\Phi_1(\zxz{1}{m}{0}{1})=\int\limits_{v(\alpha)=0}\psi(m\alpha)d^*\alpha=-\frac{1}{q-1}.
\end{equation}
This is the value of $A$.
When $i<c_1$, the proof of (2) of Claim 1 is actually similar to the proof of Lemma \ref{lemofPhi3}. We will leave this to the readers.

Now let $\pi_2$ be of form $\pi(\chi_1,\chi_2)$, where $\chi_1$ and $\chi_2$ are both ramified. By formula (\ref{Phi2_1_1}), we have the following for $v(a)=0$, $i=c_3-1$ and $v(m)\geq -c_3/2$:
\begin{align}
 \Phi_2(\zxz{a}{m}{0}{1}\zxz{1}{0}{\varpi^i}{1})=\int\limits_{v(\alpha)=0} W_2^{(c_2-1)}(\alpha) d^*\alpha=-\frac{1}{q-1}.
\end{align}
This is the value $B$.


\subsection{Type 2 occuring}
In this subsection we will consider unramified and special unramified representations. We first recall the existing work of matrix coefficients for these representations.
For unramified representations, just recall Lemma \ref{lemofMCforUNram}.

For special unramified representations, let $\sigma_n=\zxz{\varpi^n}{0}{0}{1}$, and $\omega=\zxz{0}{1}{-1}{0}$. 
\begin{lem}\label{lemofMCforSpecial}
 Let $\pi=\sigma(\chi|\cdot|^{1/2},\chi|\cdot|^{-1/2})$ be a special unramified unitary representation of $\GL_2$. It has a normalized $K_1(\varpi)-$invariant newform. The associated matrix coefficient $\Phi$ for this newform is bi-$K_1(\varpi)-$invariant and can be given in the following table for double $K_1(\varpi)-$cosets:

 \begin{tabular}{|p{2cm}|p{2cm}|p{2cm}|p{2cm}|p{2cm}|p{2cm}|}
  \hline
  g	&$\omega$	&$\sigma_n$	&$\omega\sigma_n$	&$\sigma_n\omega$	&$\omega\sigma_n\omega$\\\hline
  $\Phi(g)$	&$-q^{-1}$	&$\chi^nq^{-n}$	&$-\chi^nq^{1-n}$ &$-\chi^nq^{-1-n}$	&$\chi^nq^{-n}$	\\\hline	
 \end{tabular}

 In this table $n\geq 1$ and $\Phi(1)=1$ is not listed.
\end{lem}
This result is due to \cite{MW12}.

Now we consider unramified representations. Part (2) of Claim 1 is actually automatic in this case as $\Phi_1$ is $K-$invariant.
Let's figure out $A$ and $B$ values using Lemma \ref{lemofMCforUNram}.
Let $\Phi$ be the matrix coefficient as defined in Lemma \ref{lemofMCforUNram}, and $v(m)=-1$. Since $$\zxz{1}{m}{0}{1}=\zxz{1}{0}{1/m}{1}\zxz{m}{0}{0}{1/m}\zxz{1/m}{1}{-1}{0},$$
we have 
\begin{align}
 \Phi_1(\zxz{1}{m}{0}{1})&=\Phi_1(\zxz{m}{0}{0}{1/m})=w^{-1}_{\pi_1}\Phi(\sigma_2)\\
 &= \frac{1}{\chi_1\chi_2}(\frac{q^{-1}}{1+q^{-1}}\frac{\chi_1^2(\chi_1-\chi_2q^{-1})-\chi_2^2(\chi_2-\chi_1q^{-1})}{\chi_1-\chi_2})\notag\\
 &=\frac{1}{q+1}(\frac{\chi_1}{\chi_2}+\frac{\chi_2}{\chi_1}+1-q^{-1}).\notag
\end{align}
This is the value $A$.

On the other hand,

$$\zxz{\varpi^{c_3}}{0}{0}{1}\zxz{1}{0}{\varpi^{c_3-1}}{1}\zxz{\varpi^{-c_3}}{0}{0}{1}=\zxz{1}{0}{\varpi^{-1}}{1}=-\omega\zxz{1}{-\varpi^{-1}}{0}{1}\omega.$$
So we can similarly show that $B=\frac{1}{q+1}(\frac{\eta_1}{\eta_2}+\frac{\eta_2}{\eta_1}+1-q^{-1})$ if $\pi_2\cong \pi(\eta_1,\eta_2)$.


 
Now let $\pi_1$ be a special unramified representation of form $\sigma(\chi|\cdot|^{1/2},\chi|\cdot|^{-1/2})$, and let $\Phi$ be the matrix coefficient as given in Lemma \ref{lemofMCforSpecial}.
When $v(m)=-1$, $$\zxz{1}{m}{0}{1}=-\zxz{1}{0}{1/m}{1}\omega\zxz{1/m}{0}{0}{m}\zxz{1}{0}{1/m}{1},$$
So $$A=\Phi(\omega\zxz{1/m}{0}{0}{m})=w^{-1}_{\pi_1}\Phi(\omega\sigma_2)=\frac{1}{\chi^2}(-\chi^2q^{-1})=-q^{-1}.$$

To check (2) of Claim 1, we just need to show that when $i=0$ and $v(m)=v(a)=-c_3$, $\Phi_1(\zxz{a}{m}{0}{1}\zxz{1}{0}{1}{1})$ is at most level 1.
If $v(a+m)>v(m)$, we have 
$$\zxz{a}{m}{0}{1}\zxz{1}{0}{1}{1}=\omega\zxz{a/m}{0}{0}{m}\zxz{1}{-m/a}{0}{1}\zxz{1}{0}{\frac{a+m}{m}}{1},$$
so
\begin{equation}
\Phi_1(\zxz{a}{m}{0}{1}\zxz{1}{0}{1}{1})=\Phi(\omega\zxz{a/m}{0}{0}{m})=w_{\pi_1}^{-c_3}\Phi(\omega\sigma_{c_3}).
\end{equation}

If $v(a+m)=v(m)$, we have
$$\zxz{a}{m}{0}{1}\zxz{1}{0}{1}{1}=-\zxz{1}{0}{\frac{1}{a+m}}{1}\omega \zxz{1}{0}{0}{m}\omega\zxz{\frac{a+m}{m}}{1}{0}{\frac{a}{a+m}},$$
so
\begin{equation}
 \Phi_1(\zxz{a}{m}{0}{1}\zxz{1}{0}{1}{1})=\Phi(\omega \zxz{1}{0}{0}{m}\omega)=w_{\pi_1}^{-c_3}\Phi(\omega\sigma_{c_3}\omega).
\end{equation}
Put together, one can conclude that  $\Phi_1(\zxz{a}{m}{0}{1}\zxz{1}{0}{1}{1})$ is at most level 1 in $a$ when $v(a)=v(m)=-c_3$.

Now we compute the value $B$ for special unramified representations. 
Since
\begin{align}
 \zxz{\varpi^{c_3-1}}{0}{0}{1}\zxz{1}{0}{\varpi^{c_3-1}}{1}\zxz{\varpi^{-c_3+1}}{0}{0}{1}&=\zxz{1}{0}{1}{1}\\
 &=-\zxz{1}{1}{0}{1}\zxz{0}{1}{-1}{0}\zxz{1}{1}{0}{1} ,\notag
\end{align}
we have 
\begin{equation}
 B=\Phi_2(\zxz{\varpi^{c_3-1}}{0}{0}{1}\zxz{1}{0}{\varpi^{c_3-1}}{1}\zxz{\varpi^{-c_3+1}}{0}{0}{1})=\Phi(\omega)=-q^{-1}.
\end{equation}

\subsection{Type 3 occuring}

In this subsection, we consider the representations of form $\pi(\chi_1,\chi_2)$, where $\chi_1$ is unramified and $\chi_2$ is of level $k$. We will basically make use of Lemma \ref{Wiofbiasram}.
Let's first check part (2) of Claim 1. By (\ref{Phi1_2}), when $i<k$,
\begin{equation}
\Phi_1(\zxz{a}{m}{0}{1}\zxz{1}{0}{\varpi^i}{1})=\int\limits_{v(\alpha)\geq 0}\psi(m\alpha) W^{(i)}(a\alpha)\overline{W^{(k)}(\alpha)}d^*\alpha,
\end{equation}
Here
\begin{align}
W^{(k)}(\alpha)=\begin{cases}
q^{-\frac{1}{2}v(\alpha)}\chi_1^k(\varpi)q^{-k}\int\limits_{v(m)=-k}\chi_2(-m)\psi(-m)dm, &\text{\ if\ }v(\alpha)\geq 0,\\
0,& \text{\ otherwise,}
\end{cases}
\end{align}
and 
\begin{equation}\label{Wi_1}
W^{(i)}(a\alpha)=\chi_1^i(\varpi)\int\limits_{u\in O_\F}\chi_2(a\alpha\varpi^{-i}(1-\varpi^{k-i}u))\psi(a\alpha\varpi^{-i}(1-\varpi^{k-i}u))q^{-\frac{1}{2}v(a\alpha)-k+i}du.
\end{equation}
They are not normalized, but it turns out that this is enough. 

For fixed  $v(u)\geq 0$,  $\chi_2(a\alpha\varpi^{-i}(1-\varpi^{k-i}u))$ is of level $i-v(u)$ in $u$, $\psi(a\alpha\varpi^{-i}(1-\varpi^{k-i}u))$ is additive of level $2i-k-v(a\alpha)-v(u)$ in $u$.
For (\ref{Wi_1}) to be nonzero, we need $2i-k-v(a\alpha)-v(u)\leq i-v(u)$, that is $v(a\alpha)\geq i-k$. This is because if $2i-k-v(a\alpha)-v(u)> i-v(u)$, then the integral will be automatically zero
for fixed $v(u)<i$, and
\begin{equation}
 \int\limits_{v(u)\geq i}\chi_2(a\alpha\varpi^{-i}(1-\varpi^{k-i}u))\psi(a\alpha\varpi^{-i}(1-\varpi^{k-i}u))du=\chi_2(a\alpha\varpi^{-i})\int\limits_{v(u)\geq i}\psi(a\alpha\varpi^{-i}(1-\varpi^{k-i}u))du=0,
\end{equation}
as $v(a\alpha\varpi^{k-2i})<-i$.

Then as functions in $a$, $\chi_2(a\alpha\varpi^{-i}(1-\varpi^{k-i}u))$ is of level $k$ in $a$, and $\psi(a\alpha\varpi^{-i}(1-\varpi^{k-i}u))$ is of level $i-v(a\alpha)\leq k$ in $a$. In particular, $W^{(i)}(a\alpha)$ if of level$\leq k$ in $a$.
So (2) of Claim 1 is verified for this case.

Now for $v(m)=-1$, 
\begin{align}
\Phi_1(\zxz{1}{m}{0}{1})&=\int\limits_{v(\alpha)\geq 0}\psi(m\alpha) W^{(k)}(\alpha)\overline{W^{(k)}(\alpha)}d^*\alpha\\
&=q^{-2k}|\chi_1|^{2k}|\int\limits_{v(m)=-k}\chi_2(-m)\psi(-m)dm|^2\int\limits_{v(\alpha)\geq 0}\psi(m\alpha)q^{-v(\alpha)}d^*\alpha.\notag
\end{align}
But up to a nonzero constant, $\int\limits_{v(\alpha)\geq 0}\psi(m\alpha)q^{-v(\alpha)}d^*\alpha$ is just
$$\int\limits_{v(\alpha)\geq 0} \psi(m\alpha)d\alpha,$$
which is zero when $v(m)=-1$. So $A=0$.

Now let $\pi_2$ to be of Type 3. By (\ref{Phi2_1}),
\begin{equation}\label{Phi2ofbiasram}
 \Phi_2(\zxz{1}{0}{\varpi^{k-1}}{1})=
 \int\limits_{v(\alpha)\geq 0}  W^{(k-1)}(\alpha)\overline{W^{(k)}(\alpha)} d^*\alpha.
\end{equation}
Here 
\begin{equation}
W^{(k-1)}(\alpha)=\chi_1^{k-1}(\varpi)\int\limits_{u\in O_\F}\chi_2(\alpha\varpi^{-k+1}(1-\varpi u))\psi(\alpha\varpi^{-k+1}(1-\varpi u))q^{-\frac{1}{2}v(\alpha)-1}du.
\end{equation}
As functions in $\alpha$, $\chi_2(\alpha\varpi^{-k+1}(1-\varpi u))$ is multiplicative of level $k$, but $\psi(\alpha\varpi^{-k+1}(1-\varpi u))$ is of level $k-1-v(\alpha)\leq k-1$. Then the integral in (\ref{Phi2ofbiasram}) has to be zero as
$W^{(k-1)}(\alpha)$ doesn't have any level 0 components. So $B=0$.

\section{Conclusion}

Using Theorem \ref{thmlocalint}, we get the following lower bound for local integrals:
\begin{prop}\label{propoflocallowerbd}
\begin{equation}
\prod_{v\in S} I_v(f_1, \rho(a([\mathcal{N}]))f_2, f_3)\gg N^{-1-\epsilon}
\end{equation}
\end{prop}
\begin{proof}The case when $\pi_{1,v}$ and $\pi_{2,v}$ are unramified and $\pi_{3,v}$ is special unramified representation is considered in \cite{MW12}. For the rest places covered in Theorem \ref{thmlocalint}, the local L-factors are trivial.
 Then essentially the factors $\frac{1}{(q+1)q^{c_3-1}}$ will contribute to the part $N^{-1}$. 
 
 Any fixed constant bound can be absorbed into $N^{-\epsilon}$. This is to say if we have, for example, local inequalities
 $$I_v\geq C\frac{1}{q^{c_3}},$$
 then we are safe to take a product and claim that
 $$\prod_{v\in S} I_v\gg N^{-1-\epsilon}.$$
 This is because $C$ will be finally strictly greater than $\frac{1}{q^{c_3\epsilon}}$, when either $q\rightarrow +\infty$ or $c_3\rightarrow +\infty$.
 
 So in particular we don't have to worry about factors like $\frac{1}{1+q^{-1}}$ and $\frac{1}{\zeta_v(2)}=1-q^{-2}$. What remains to be checked is that $A$ and $B$ should be bounded away from 1. This is clear from Theorem \ref{thmlocalint} for Types 1 and 3 and also special unramified representations.
 For unramified representations, we have
 $$|1-A|=\frac{|q+q^{-1}-\frac{\chi_1}{\chi_2}-\frac{\chi_2}{\chi_1}|}{q+1}\geq \frac{q+q^{-1}-(|\frac{\chi_1}{\chi_2}|+|\frac{\chi_2}{\chi_1}|)}{q+1}.$$
 This is clearly bounded below if the representation is tempered. When it's not tempered, we need to use the bound towards Ramanujan $\alpha$. So
 $$|1-A|\geq \frac{q+q^{-1}-(q^{2\alpha}+q^{-2\alpha})}{q+1},$$
which is clearly bounded below.
\end{proof}
\begin{cor}\label{coroflocalepsiloninS}
 For $v\in S$, we always have
 $$\epsilon_v(\Pi_v,1/2)=1.$$
\end{cor}
\begin{proof}
 The claim just follows from  Prasad's work and that the local integrals in Theorem \ref{thmlocalint} are nonzero.
\end{proof}

\begin{theo}\label{thmmain}
 Let $\pi_i, i=1,2,3$ be three unitary cuspidal automorphic representations of $\GL_2$, such that 
 \begin{equation}
  \prod_iw_{\pi_i}=1.
 \end{equation}

 Fix $\pi_1$ and $\pi_2$, and let $\pi_3$ vary with changing finite conductor $\mathcal{N}$ and $N=\text{Nm}(\mathcal{N})$. Suppose that the infinity component of $\pi_3$ is still bounded. Then
 \begin{equation}
  L(\pi_1\otimes\pi_2\otimes\pi_3,1/2)\ll N^{1-1/12}, \text{\ as\ }N\rightarrow \infty.
 \end{equation}

\end{theo}
\begin{proof}
 According to \cite{Iw90},
\begin{equation}
 L(\Pi, Ad, 1)\ll C(\Pi)^\epsilon,
\end{equation}
where the implicit constant depends continuously on the Langlands parameter of the infinity component. In particular when $\pi_1$ $\pi_2$ are fixed and $\pi_3$ has bounded infinity component,
\begin{equation}
 L(\Pi, Ad, 1)\ll N^\epsilon.
\end{equation}

Then to prove the theorem, one just need to apply Proposition \ref{propofglobalupbd} and Proposition \ref{propoflocallowerbd}, and also Remark \ref{remofdifferentN} and Remark \ref{remofplacesoutsideS}.
\end{proof}

\appendix
\section{Bound of global matrix coefficient}
Here we will prove Proposition \ref{propofboundglobalMC}. We record it here again. 
Let $\D$ be a global quaternion algebra. Let $\rho$ denote the right regular representation of $\D^*(\A)$ on $L^2(Z_{\A}\D^*(\F)\backslash \D^*(\A))$. 
Let $F_1$, $F_2\in L^2(Z_{\A}\D^*(\F)\backslash \D^*(\A))$ be two rapidly decreasing and $K-$finite automorphic forms which don't have 1-dim components in their spectrum decomposition. Implicitly the center acts on $F_i$ trivially.
 Let $S$ be a finite set of non-archimedean places. We assume that $\D$ is locally the matrix algebra at the places in $S$.  
Let $K_S=\prod_{v\in S}K_v$ and 
$$K_{i,S}=\prod_{v\in S}K_{i,v}\text{\ \ }K_i=\prod_{v\text{\ finite}}K_{i,v},$$ 
where $K_{i,v}$ stabilizes the local component of $F_i$ at $v$. Let $\mathcal{N}=\prod_v\varpi_v^{e_v}$ for $e_v\geq 0$, and $N=\text{Nm}(\mathcal{N})$. Define the matrix 
 $$a([\mathcal{N}])=\prod_v\zxz{\varpi^{-e_v}}{0}{0}{1},$$
 which can be naturally thought of as an element of $\D^*(\A)$.
\begin{prop}
With the setting as above, we have
 \begin{equation}
  |\int\limits_{Z_{\A}\D^*(\F)\backslash \D^*(\A)}F_1(g)\rho(a([\mathcal{N}]))F_2(g)dg|\ll_{\epsilon,\F}[K_S:K_{1,S}]^{1/2}[K_S:K_{2,S}]^{1/2}N^{\alpha-1/2+\epsilon}||F_1||_{L^2}||F_2||_{L^2}.
 \end{equation}

\end{prop}

First of all, the case when $\D$ is a division algebra is actually simple to prove. This is because there is no continuous spectrum for $L^2(\Z_\A\D^*(\F)\backslash \D^*(\A))$. 
 When the automorphic forms are from a single cuspidal representation, one just need to take a product of local bounds in Lemma \ref{localboundmatrixcoeff} following Remark \ref{remoftakingproductofbound}. 
In general we consider the spectrum decomposition for $F_i$: 
$$F_i=\sum\limits_{\pi}\sum\limits_{f\in \mathcal{B}(\pi)}<F_i,f>f,$$
where $\mathcal{B}(\pi)$ is an orthonormal basis under the unitary pairing for a cuspidal automorphic representation $\pi$.
If $F_i$ is invariant under $K_{i}$, then its cuspidal component in $\pi$:
$$\sum\limits_{f\in \mathcal{B}(\pi)}<F_i,f>f$$ 
is also invariant under $K_{i}$. This is true because of Plancherel Theorem.
As a result, one can apply the argument for the previous case for each such component, and then use Cauchy-Schwarz inequality.

When $\D$ is the matrix algebra, one can argue similarly if $F_1$, $F_2$ have only cuspidal spectrums. But in general, they can have continuous spectrums. Intuitively the continuous spectrums shouldn't mess things up because they are related to Eisenstein series defined by unitary characters,  and they look like tempered representations locally. 

To argue more strictly, let's first recall some results. 

\subsection{The spectrums of $L^2(\Z_\A\GL_2(\F)\backslash\GL_2(\A))$ and the Plancherel formula}

For a more detailed reference of this subsection, see Section 2.2 in \cite{MV10}.

Let $\mathcal{X}$ denote the set of pairs $(M,\sigma)$, where $M$ is a $\F-$Levi subgroup of a $\F-$parabolic subgroup (containing a maximally $\F-$split torus $T$), and $\sigma$ is an irreducible automorphic representation of $M$ naturally embedded in $L^2(\Z_M M(\F)\backslash M(\A))$. 
For $\GL_2$ there are two cases. When $M$ is the whole group $\GL_2$, then $\sigma$ is just a cuspidal representation. When $M$ is the torus $T$, then $\sigma$ is actually a unitary character on the torus.

We can equip $\mathcal{X}$ with a measure in the following way: we write 
\begin{equation}
 \mathcal{X}=\bigsqcup\limits_{M}\mathcal{X}_M,
\end{equation}
indexed by levis containing $T$. We require that for any continuous assignment of $\chi\in \mathcal{X}_M$ to $f_\chi$ in the underlying space of $\chi$,
\begin{equation}
 \int\limits_{\Z_M(\F)\backslash M(\A)}|\int\limits_{\chi}f_\chi d\chi|^2=\int\limits_{\chi}||f_\chi||^2_\sigma d\chi.
\end{equation}
This uniquely specifies a measure $d\chi$ on $\mathcal{X}_M$, and so also on $\mathcal{X}$. Note when $\sigma $ is a unitary character on the torus, $||\cdot||_\sigma$ is just the usual absolute value.

$(M,\sigma)$ is said to be equivalent to $(M', \sigma')$ if there exists $\omega$ in the normalizer of $T$ with $Ad(\omega)M=M'$ and $Ad(\omega)\sigma=\sigma'$. There is a natural quotient measure on $\mathcal{X}/\sim$.

For $\chi=(M,\sigma)\in \mathcal{X}$, we donte by $\mathcal{I}(\chi)$ the unitary induced representation $\text{Ind}_{P(\A)}^{G(\A)}\sigma,$ where $P$ is any parabolic subgroup containing $M$. One can define a unitary pairing on $\mathcal{I}(\chi)$ by
\begin{equation}
 <f_1,f_2>_{\text{Eis}}=\int\limits_{K}<f_1(k),f_2(k)>_\sigma dk,
\end{equation}
where $K$ is equipped with Haar probability measure. When $M=T$ and $\sigma$ is just a unitary character of $T$, this pairing is just
\begin{equation}\label{formulaofglobalEispairing}
 <f_1,f_2>_{\text{Eis}}=\int\limits_{K}f_1(k)\overline{f_2(k)}dk,
\end{equation}
which is directly a product of local integrals.
With this pairing, one can talk about orthonormal basis for $\mathcal{I}(\chi)$. We will denote such an orthonormal basis by $\mathcal{B}(\chi)$.

For any element $\varphi\in \mathcal{I}(\chi)$, one can define the corresponding Eisenstein series by just averaging over $P(\F)\backslash G(\F)$ and analytic continuation. We will denote the corresponding Eisenstein series by $E_{\chi,\varphi}$.

For rapidly decreasing functions we have the following Plancherel formulae
\begin{equation}\label{formofspectrumdecomp}
F=\int\limits_{\chi\in \mathcal{X}/\sim}\sum\limits_{\varphi\in \mathcal{B}(\chi)}<F,E_{\chi,\varphi}>E_{\chi,\varphi}d\chi, 
\end{equation}

\begin{equation}\label{formulaofplancherel}
 <F_1, F_2>=\int\limits_{\chi\in \mathcal{X}/\sim}\sum\limits_{\varphi\in \mathcal{B}(\chi)}<F_1,E_{\chi,\varphi}>\overline{<F_2,E_{\chi,\varphi}>}d\chi.
\end{equation}

\subsection{Proof continued}
Now we shall finish the proof of Proposition \ref{propofboundglobalMC}. As we've already proved the proposition for the cuspidal part, we can assume from now on that $F_1$ and $F_2$ have only continuous spectrums and then use Cauchy-Schwartz inequality to piece together.

By (\ref{formulaofplancherel}), 
\begin{align}\label{formulaofglobalmatrixcoeff1}
 <F_1, \rho(a([\mathcal{N}]))F_2>&=\int\limits_{\chi\in \mathcal{X}_T}\sum\limits_{\varphi\in \mathcal{B}(\chi)}<F_1,E_{\chi,\varphi}>\overline{<\rho(a([\mathcal{N}]))F_2,E_{\chi,\varphi}>}d\chi\\
&= \int\limits_{\chi\in \mathcal{X}_T}\sum\limits_{\varphi\in \mathcal{B}(\chi)}<F_1,E_{\chi,\varphi}>\overline{<F_2,\rho(a([\mathcal{N}])^{-1})E_{\chi,\varphi}>}d\chi.\notag
\end{align}
Note that 
$$\rho(a([\mathcal{N}])^{-1})E_{\chi,\varphi}=E_{\chi,\rho(a([\mathcal{N}])^{-1})\varphi},$$
where $\rho(a([\mathcal{N}])^{-1})\varphi$ is still in $\mathcal{I}(\chi)$ and is $K-$finite. In particular, we can decompose $\rho(a([\mathcal{N}])^{-1})\varphi$ using the orthonormal basis $\mathcal{B}(\chi)$: 
\begin{equation}
 \rho(a([\mathcal{N}])^{-1})\varphi=\sum\limits_{\varphi'\in \mathcal{B}(\chi)}<\rho(a([\mathcal{N}])^{-1})\varphi,\varphi'>_{\text{Eis}}\varphi'.
\end{equation}
Note this is a purely local argument, and the sum on the right hand side is just a finite sum if we pick the basis properly.
Correspondingly, 
\begin{equation}
 \rho(a([\mathcal{N}])^{-1})E_{\chi,\varphi}=\sum\limits_{\varphi'\in \mathcal{B}(\chi)}<\rho(a([\mathcal{N}])^{-1})\varphi,\varphi'>_{\text{Eis}}E_{\chi, \varphi'}.
\end{equation}
Now the part associated to $\chi$ in (\ref{formulaofglobalmatrixcoeff1}) becomes
\begin{align}
&\text{\ \ \ }\sum\limits_{\varphi\in \mathcal{B}(\chi)}<F_1,E_{\chi,\varphi}>\overline{<F_2,\rho(a([\mathcal{N}])^{-1})E_{\chi,\varphi}>}\\
&=\sum\limits_{\varphi,\varphi'}<F_1,E_{\chi,\varphi}><\rho(a([\mathcal{N}])^{-1})\varphi,\varphi'>_{\text{Eis}}\overline{<F_2,E_{\chi, \varphi'>}}\notag\\
&=<\sum\limits_{\varphi}<F_1,E_{\chi,\varphi}>\rho(a([\mathcal{N}])^{-1})\varphi,\sum\limits_{\varphi'}<F_2,E_{\chi, \varphi'}>\varphi'>_{\text{Eis}}\notag\\
&=<\sum\limits_{\varphi}<F_1,E_{\chi,\varphi}>\varphi,\rho(a([\mathcal{N}]))\sum\limits_{\varphi'}<F_2,E_{\chi, \varphi'}>\varphi'>_{\text{Eis}}.\notag
\end{align}

For each $ \sum\limits_{\varphi\in \mathcal{B}(\chi)}<F_i, E_{\chi,\varphi}>\varphi$ in the expression above we have the following lemma:
\begin{lem}
 If $F_i$ is $K_i-$invariant, then for any cuspidal datum $\chi$ associated to $T$,
 \begin{equation}\label{formulaofEiscomponent}
  \sum\limits_{\varphi\in \mathcal{B}(\chi)}<F_i, E_{\chi,\varphi}>\varphi
 \end{equation}
is also $K_{i}-$ invariant.
\end{lem}
\begin{proof}
First (\ref{formulaofEiscomponent}) is independent of the choice of the basis. In particular one can pick an orthonormal basis for $\mathcal{I}(\chi)^{K_{i}}$ first, then extend it to an orthonormal basis for $\mathcal{I}(\chi)$. To prove the lemma, it is enough to show for this basis that if $\varphi\in \mathcal{B}(\chi)-\mathcal{B}(\mathcal{I}(\chi)^{K_{i}})$,
 then $<F_i,E_{\chi, \varphi}>=0$. 
 
 By definition and the standard unfolding technique,
\begin{align}
<F_i,E_{\chi,\varphi}>&=\int\limits_{g\in Z_{\A}\GL_2(\F)\backslash \GL_2(\A)}F_i(g)\overline{E_{\chi,\varphi}(g)}dg\\
&=\int\limits_{g\in Z_{\A}B(\F)\backslash \GL_2(\A)}F_i(g)\overline{\varphi(g)}dg\notag\\
&=\int\limits_{g\in Z_{\A}B(\F)\backslash \GL_2(\A)/K_i}F_i(g)\int\limits_{k\in 
K_i}\overline{\varphi(gk)}dkdg.\notag
\end{align}
We'd like to see for every $g$, whether the following integral is zero or not:
\begin{equation}\label{formulaofEisparing}
 \int\limits_{K_i}\overline{\varphi(gk)}dk=\prod\limits_{v}\int\limits_{K_{i,v}}\overline{\varphi_v(g_vk_v)}dk_v.
\end{equation}
Now fix $g$. For every $v$, consider the double coset decomposition
$$\GL_2(\F_v)=\coprod\limits_{a_i}Ba_iK_{i,v},$$
where $\{a_i\}$ is a set of double coset representatives. Locally we can write $g_v=b_va_{g(v)}k_v'$, where $a_{g(v)}\in\{a_i\}$. Note for almost all places, $K_{i,v}=K_v$ and $g_v\in K_v$, so $b_v$ and $a_{g(v)}$ will be trivial there.
Then (\ref{formulaofEisparing}) is zero if and only if
\begin{equation}\label{formulaofEisparing2}
 \prod\limits_{v}\int\limits_{K_{i,v}}\overline{\varphi_v(a_{g(v)}k_v)}dk_v
\end{equation}
is zero.

But for every fixed $g$ this integral is the same (up to a nonzero constant) as the pairing $<\cdot,\cdot>_\text{Eis}$ in $\mathcal{I}(\chi)$ between $\varphi$ and another element $\varphi'\in \mathcal{I}(\chi)$ whose local component at $v$ is singly supported on $Ba_{g(v)}K_{i,v}$. This element $\varphi'$ is clearly in $\mathcal{I}(\chi)^{K_i}$, so by the choice of $\varphi$, 
\begin{equation}
 <\varphi',\varphi>_\text{Eis}=0.
\end{equation}
So (\ref{formulaofEisparing}) is zero and $<F_i,E_{\chi,\varphi}>=0$ for $\varphi\in \mathcal{B}(\chi)-\mathcal{B}(\mathcal{I}(\chi)^{K_{i}})$.
\end{proof}
Recall the pairing $<\cdot,\cdot>_\text{Eis}$ is directly a product of local pairings, and for the local pairing, we can use Lemma \ref{localboundmatrixcoeff} to bound the local matrix coefficient. 
Note the local components of $\mathcal{I}(\chi)$ are always tempered.
By taking a product and using the result that $\sum\limits_{\varphi\in \mathcal{B}(\chi)}<F_i, E_{\chi,\varphi}>\varphi$ is $K_{i}-$ invariant, we can get
\begin{align}
&\text{\ \ \ }\sum\limits_{\varphi\in \mathcal{B}(\chi)}<F_1,E_{\chi,\varphi}>\overline{<F_2,\rho(a([\mathcal{N}])^{-1})E_{\chi,\varphi}>}\\
&\ll_{\epsilon,\F} [K_S:K_{1,S}]^{1/2}[K_S:K_{2,S}]^{1/2} N^{-1/2+\epsilon}||\sum\limits_{\varphi}<F_i, E_{\chi,\varphi}>\varphi ||_\text{Eis} ||\sum\limits_{\varphi}<F_2, E_{\chi,\varphi}>\varphi ||_\text{Eis}\notag \\
&\leq [K_S:K_{1,S}]^{1/2}[K_S:K_{2,S}]^{1/2} N^{\alpha-1/2+\epsilon}(\sum\limits_{\varphi}|<F_1, E_{\chi,\varphi}>|^2)^{1/2} (\sum\limits_{\varphi}|<F_2, E_{\chi,\varphi}>|^2)^{1/2},\notag
\end{align}
for any bound towards Ramanujan Conjecture $\alpha$. Finially when we do the integral in cuspidal datum $\chi$, just apply Cauchy-Schwartz inequality.


\begin{thebibliography}{11}
\bibitem{Bump}
D.Bump, \textit{Automorphic Forms and Representations}. 
Cambridge Studies in Advanced Mathematics,
vol. 55, Cambridge University Press, Cambridge, 1997.

\bibitem{BB10}
V.Blomer and F.Brumley,
\textit{On the Ramanujan Conjecture over number fields},
Ann. of Math. (2) 174 (2011), no. 1, 581-605.

\bibitem{CHH88}
M.Cowling, U.Haagerup and R.Howe,
\text{Almost $L^2$ matrix coefficients},
J.reine angew. Math. 387(1988), 97-110.

\bibitem{Garrett}
P.B.Garrett, \textit{Docomposition of Eisenstein series:Rankin triple products}. 
Ann. of Math. (2)125 (1987), 209-235.

\bibitem{H&K91}
M.Harris and S.S.Kudla,
\textit{The central critical value of a triple product L-function}. 
Annals of Math. 133 (1991),
605-672.

\bibitem{H&K04}
M.Harris and S.S.Kudla,
\textit{On a conjecture of Jacquet}. Contributions to automorphic forms, geometry, and number theory, 355-371, Johns Hopkins Univ. Press, Baltimore, MD, 2004.

\bibitem{YH13}
Y.Hu,
\textit{Cuspidal part of an Eisenstein series restricted to an index 2 subfield,}
arXiv:1309.7467.

\bibitem{Ichino}
A.Ichino,
\textit{Trilinear forms and the central values of triple product L-functions}.
Duke Math. J. 145 (2008), no. 2, 281-307.

\bibitem{Iw90}
H.Iwaniec,
\textit{Small eigenvalues of Laplacian for $\Gamma_0(N)$}.
Acta Arith., 56(1):65-82, 1990.

\bibitem{JL70}
H.Jacquet and R.P.Langlands, \textit{Automorphic forms on $\GL(2)$}.  Lecture Notes in Mathematics, Vol. 114. Springer-Verlag, Berlin-New York, 1970. vii+548 pp.

\bibitem{Kim}
H. Kim and P.Sarnak, 
\textit{Refined estimates towards the Ramanujan and Selberg Conjectures},
J. Amer. Math. Soc. 16 (2003), 139-183, Appendix to H. Kim, \textit{Functoriality
for the exterior square of GL(4) and symmetric fourth of GL(2)}.


\bibitem{K&R}
S.S.Kudla and S.Rallis, \textit{A regularized Siegel-Weil formula: The first term identity}.
Ann. of Math. (2)140(1994), 1-80.

\bibitem{NPS12}
P.Nelson, A.Pitale and A.Saha,
\textit{Bounds for Rankin-Selberg integrals and quantum unique ergodicity for powerful levels},
arXiv:1205.5534.

\bibitem{ps}
I.Piatetski-Shapiro and S.Rallis,
\textit{ Rankin triple L functions}. Compositio Math. 64
(1987), 31-115.

\bibitem{Prasad}
D. Prasad, \textit{Trilinear forms for representations of $\GL(2)$ and local $\epsilon$-factors}.
Compositio Math. 75 (1990), 1-46.

\bibitem{Tu}
J.B.Tunnell,
\textit{Local $epsilon$-factors and characters of $\GL(2)$}. 
Amer. J. Math. 105 (1983), no. 6, 1277-1307.

\bibitem{Watson}
T.C.WATSON, \textit{Rankin triple products and quantum chaos}. Thesis (Ph.D.), Princeton University. 2002. 81 pp. 

\bibitem{Yo77}
H.Yoshida, \textit{On extraordinary representations of GL2,} Algebraic number theory, Japan
Soc. for the promotion of science, Tokyo (1977).

\bibitem{MV10}
P.Michel and A.Venkatesh, \textit{The subconvexity problem for $\GL_2$,}
arxiv.org/pdf/0903.3591.

\bibitem{AV10}
A.Venkatesh, \textit{Sparse equidistribution problems, period bounds and subconvexity}. 
Ann. of Math., 172(2010), 989-1094.

\bibitem{MW12}
M.Woodbury, \textit{Trilinear forms and subconvexity of the tiple product L-function},
submitted.


\end{thebibliography}
\end{document}